\documentclass[12pt]{article}

\usepackage{amsthm,amsmath,amssymb}

\usepackage{graphicx}
 \usepackage{epsfig}
  \usepackage{latexsym, epsfig, psfrag,eepic,colordvi,bm}
  \usepackage{graphics,color}
  \usepackage{graphics}
  \usepackage{graphicx}
  \usepackage{epsfig}
  \usepackage[all]{xy}
\usepackage{cancel}
  \usepackage{amssymb}
  \usepackage{amsthm,amsmath}

  \usepackage{latexsym, epsfig, psfrag,eepic,colordvi,bm}
  \usepackage{graphics,color}
  \usepackage{amsmath,amsfonts,amssymb,amscd}
  \usepackage[all]{xy}
  \usepackage{epstopdf}

  \usepackage{amsthm,amsmath,amssymb}

  \usepackage{graphicx}
  \usepackage{amsthm,amsmath,amssymb}

  \usepackage{graphicx,epsfig}
\usepackage[colorlinks=true,citecolor=black,linkcolor=black,urlcolor=blue]{hyperref}

\usepackage[all]{xy}
\usepackage{graphicx}
\usepackage{mathptmx}

\usepackage{latexsym, epsfig,cite, psfrag,eepic,colordvi,bm}
\usepackage{graphics,color,graphicx}
\usepackage[justification=centering]{caption}
\usepackage{amsmath,amsfonts,amssymb,amscd}
\usepackage[all]{xy}

\theoremstyle{plain}
\newtheorem{theorem}{Theorem}[section]

\newtheorem{lemma}{Lemma}[section]
\newtheorem{corollary}{Corollary}[section]
\newtheorem{proposition}{Proposition}[section]

\theoremstyle{definition}
\newtheorem{definition}[theorem]{Definition}

\theoremstyle{remark}
\newtheorem{remark}[theorem]{Remark}

\newcommand{\genstirlingI}[3]{%
	\genfrac{[}{]}{0pt}{#1}{#2}{#3}%
}

\newcommand{\stirlingI}[2]{\genstirlingI{}{#1}{#2}}

\date{}

\title{Hurwitz numbers with completed cycles and Gromov--Witten theory relative to at most three points}

\author{Ricky X. F. Chen~\footnote{ORCID: 0000-0003-1061-3049}, \quad Zhen-Ran Wang\\
	\small School of Mathematics, Hefei University of Technology \\[-0.8ex]
	\small Hefei, Anhui 230601, P.~R.~China\\[-0.8ex]
	\small\tt xiaofengchen@hfut.edu.cn
}

\begin{document}

\maketitle

\tableofcontents

\newpage
\begin{abstract}

Hurwitz numbers with completed cycles are standard Hurwitz numbers with simple branch points
replaced by completed cycles. In fact, simple branch points correspond to completed $2$-cycles.
Okounkov and Pandharipande have established the remarkable GW/H correspondence, saying that
the stationary sectors of the Gromov--Witten theory relative to $r$ points equal Hurwitz numbers
with $r$ branch points besides the completed cycles. However, from the viewpoint of computation,
known results for Hurwitz numbers (standard or with completed cycles) are mainly for $r\leq 2$.
It is hard to obtain explicit formulas and then discuss the structural properties for the cases $r>2$.
In this paper, we obtain explicit formulas for the case $r=3$ and uncover a number of structural properties
of these Hurwitz numbers. For instance, we discover a piecewise polynomiality with respect to the orders of the completed
cycles in addition to the parts of the profiles of branch points as usual, we show that certain hook-shape Hurwitz numbers
are building blocks of all our Hurwitz numbers, and we prove an analogue of the celebrated $\lambda_g$-conjecture.


\end{abstract}
\section{Introduction}\label{sec1}

\subsection{Hurwitz theory}

In 1891, Hurwitz~\cite{hurwitz} introduced the problem of enumerating 
non-equivalent ramified coverings of a Riemann surface $M$ of genus $h$
by a Riemann surface $M'$ of genus $g$. The desired numbers are now called Hurwitz numbers.
Here we only focus on
the coverings of the Riemann sphere $\mathbb{S}^2$ (or $\mathbb{CP}^1$).
A ramified covering of degree $d$ from $M$ to $\mathbb{S}^2$ is a nonconstant meromorphic function $f: M\rightarrow \mathbb{S}^2$
such that all but a finite number of points in the sphere have $d$ distinct preimages.
These points with less than $d$ preimages are called branch points of the covering.
The preimages of a branch point are called poles.
Suppose a branch point $p$ has $n$ preimages $z_1, z_2, \ldots, z_n$ and the point $z_i$ has a multiplicity of $\alpha_i$.
Then, $$\alpha_1+\alpha_2+\cdots +\alpha_n=d,$$
 and $z_i$ is called a pole of order $\alpha_i$ while $p$
is called a branch point of ramification profile $\alpha=(\alpha_1, \alpha_2, \ldots, \alpha_n)$.
A simple branch point is a branch point having $d-1$ preimages.
We view two ramified coverings $f$ and $f'$ as equivalent if there exists a homomorphism $h: M\rightarrow M$
such that $f=f'\circ h$.

Hurwitz himself provided a formula for the numbers corresponding to a single nonsimple branch point besides simple ones and $g=0$ without a complete proof, and this formula was rediscovered by Goulden and Jackson~\cite{GJ97}.
These studied numbers were and will be referred to as single Hurwitz numbers.
The approach of Goulden and Jackson is based on the equivalent permutation factorization formulation, that is, the desired number is the same as
the number of ordered factorizations of a permutation of $d$ elements and cycle-type $\alpha$ into
$t$ transpositions for some number $t$ where the transpositions act transitively on the $d$ elements. This approach is believed to be close to
the original Hurwitz way of reasoning~\cite{recon96}.
Afterwards, explicit formulas of these single Hurwitz numbers $\mathsf{H}^g_{\alpha}$ for $g=1$ and $g=2$ were obtained in
Goulden and Jackson~\cite{GJ99,GJ99a}, Goulden, Jackson and Vainshtein~\cite{GJV00}.
A variety of Hurwitz numbers with two nonsimple branch points have been also extensively studied subsequently.

\subsection{Hodge integrals and Gromov--Witten theory}
The single Hurwitz numbers $\mathsf{H}^g_{\alpha}$ have also been found very closely connected with the geometry of the moduli spaces of curves and
diverse physical theories. In particular, the celebrated ELSV formula~\cite{ELSV1,ELSV} has established a remarkable connection with the Hodge integrals:
\begin{align}
\mathsf{H}^g_{\alpha}= t! \prod_{i=1}^n \frac{\alpha_i^{\alpha_i}}{\alpha_i !} \int_{\overline{{M}}_{g,n}} \frac{1-\lambda_1+\lambda_2-
\cdots \pm \lambda_g}{(1-\alpha_1 \psi_1)\cdots (1-\alpha_n \psi_n)},
\end{align}
where $\overline{{M}}_{g,n}$ is the moduli space of stable curves of genus $g$ with $n$ marked points,
$\lambda_i$ is a certain codimension $i$ class, and $\psi_k$ is a Chern class.
We refer to~\cite{ELSV1,ELSV} for precise definitions of these objects.
Succinctly from the ELSV formula, the quantity
$$
\mathsf{H}^g_{\alpha}\bigg/t! \prod_{i=1}^n \frac{\alpha_i^{\alpha_i}}{\alpha_i !}
$$
is a polynomial in $\alpha_i$'s, settling
a conjecture in~\cite{GJV00}.

Along these lines, a variety of Hurwitz numbers such as 
\begin{itemize}
\item double Hurwitz numbers~\cite{GJV05,DL22,wedge15},
\item double Hurwitz numbers with completed cycle~\cite{3-cycle,complete11}, 
\item monotone
Hurwitz numbers~\cite{monotone} and,
\item  orbifold Hurwitz numbers~\cite{orbiford}
\end{itemize}
 have been extensively investigated, and finding other connections such as searching their Hodge integral counterparts has attracted
lots of attention as well.
 These existing Hurwitz numbers have been also studied in the framework of KP hierarchy, see e.g.~\cite{GJ08,bdks,kaz-lan15} (in particular,
\cite{bdks} may cover all existing Hurwitz numbers). 

Hurwitz numbers are also related to the Gromov--Witten theory.
For instance, Li, Zhao and Zheng~\cite{lzz00}
obtained a recurrence for single Hurwitz numbers via the relative Gromov--Witten theory.
But, our focus will be on another connection introduced below.
The Gromov--Witten theory of target curve $X$ relative to $m$ points of monodromy $\eta_i$'s is defined via intersection in the moduli space
$\overline{{M}}_{g,s}(X,\eta_1,\ldots, \eta_m)$ parametrizing genus $g$, $s$-pointed relative stable maps
with monodromy $\eta_i \vdash d$, i.e., a partition of $d$, (at the point $q_i$). The stationary sector of the relative Gromov--Witten theory is defined to be
the integral of products of descendent classes against the virtual fundamental class, that is,
\begin{align}
 \left\langle \prod_{i=1}^{s}\tau_{k_i}(\omega),\eta_1,\ldots, \eta_m\right \rangle^{X}_{g,d} :=	\int_{[\overline{{M}}_{g,s}(X,\eta_1,\ldots, \eta_m)]^{vir}} \prod_{i=1}^s \psi_i^{k_i} ev_i^*(\omega),
\end{align}
where $\psi_i^{k_i} ev_i^*(\omega)$ are the descendent classes and $\omega$
denotes the Poincar\'{e} dual of the point class.
We will be only interested in the case $X=\mathbb{CP}^1 $ in this paper.

Okounkov and Pandharipande~\cite{ok-pa} established a remarkable correspondence
between the Hurwitz theory and the Gromov--Witten theory.
Specifically, they proved that the stationary sector can be computed
via Hurwitz numbers with completed cycles by replacing simple branch points
with completed cycles in standard Hurwitz numbers.
Readers are invited to consult~\cite{lzz00,ok-pa} for precise definitions.

From the perspective of computation, here is a brief summary of the state of the art:
\begin{itemize}
\item existing results on standard Hurwitz numbers may be interpreted as $\left\langle \tau_{1}^s(\omega),\eta_1,\eta_2\right \rangle^{X}_{g,d} $
for some $\eta_1$ and $\eta_2$,
\item certain generating function for the stationary sectors $\left\langle \prod_{i=1}^{s}\tau_{k_i}(\omega),\eta_1,\eta_2\right \rangle^{X}_{g,d} $ for arbitrary $\eta_1$ and $\eta_2$ has been given in Okounkov and Pandharipande~\cite{ok-pa} and led to a relatively simpler
formula only for
$\left\langle \tau_{k}(\omega),\eta_1,\eta_2\right \rangle^{X}_{g,d} $,
\item the stationary sectors of the form
$\left\langle \tau_{r}^s(\omega),\eta_1,\eta_2\right \rangle^{X}_{g,d} $ for an arbitrarily fixed $r$ and arbitrary $\eta_1$ and $\eta_2$ have been given
an expression (not elementary) in Shadrin, Spitz and Zvonkine~\cite{complete11} employing the operator formalism over the wedge space, and
\item the stationary sectors of the form
$\left\langle \tau_{2}^s(\omega),\eta_1,\eta_2\right \rangle^{X}_{g,d} $ for $\eta_1=(d)$ and an arbitrary $\eta_2$ have been explicitly determined in Nguyen~\cite{3-cycle}.
\end{itemize}
The above list may not be complete, but the point is that,
it is a difficult task to compute the stationary sectors
relative to more than two points and with arbitrary combination of descendent classes (i.e., $k_i$'s).
In this paper, we make some progress in this regard.

\subsection{Our main results}

Informally, our main results are summarized below.
\begin{itemize}
\item We prove an explicit formula for the stationary sectors of the form 
$$
\sum_{\mbox {$\eta_3$ with $m$ parts}} \left\langle \prod_{i=1}^{s}\tau_{k_i}(\omega),\eta_1,\eta_2, \eta_3\right \rangle^{X}_{g,d} ,
$$
where $k_i$'s are arbitrary, $\eta_1=(d)$ and $\eta_2$ is arbitrary.
\item We show that any of the above general ``triple" stationary sectors can be written as a linear combination of the ``hook-shape" stationary sectors, i.e.,
those of the form  
$$
\sum_{\mbox {$\eta_3$ with $m$ parts}} \left\langle \prod_{i=1}^{s}\tau_{k_i}(\omega),(d),(d-k,1,\ldots, 1), \eta_3\right \rangle^{X}_{g,d} .
$$
Moreover, if $k_1+\cdots +k_s =d'$, then the last quantity can be further written as a linear combination of the stationary sectors of the form
$$
\sum_{\mbox {$\eta_3$ with $m$ parts}} \left\langle \tau_{1}^{k'}(\omega) \tau_{d'-k'}(\omega),(d),(d-k,1,\ldots, 1), \eta_3\right \rangle^{X}_{g,d} 
$$
which are easy to compute.
\item If we view the above general triple stationary sectors as a function of $k_i$'s and the parts of $\eta_2$, then the function is shown to be a piecewise polynomial in $k_i$'s and the parts of $\eta_2$.
    Polynomiality and piecewise polynomiality have been extensively discussed for single and double Hurwitz numbers.
    \item Note that the case $m=d$ reduces to certain stationary sectors relative to two points. As such, we recover and generalize a number of existing results on double Hurwitz numbers.
\end{itemize}

\subsection{Plan of the paper}
The organization of the paper is as follows.
In Section~\ref{sec2}, we review symmetric group characters and a recursion generalizing the Frobenius identity.
In Section~$3$, we review shifted symmetric functions and completed cycles.
We obtain an explicit formula for certain triple Hurwitz numbers with completed
cycles and discuss various consequences in Section~$4$.
In Section~$5$, we first introduce a theory of an algebra associated to integer partitions.
We then prove that any triple Hurwitz number can be decomposed into hook-shape Hurwitz numbers which
are much easier to compute.
In Section~$6$, we discuss polynomiality of triple and double Hurwitz numbers.

\section{Group characters and a novel recursion}\label{sec2}

\subsection{Preliminaries}

Let $\mathfrak{S}_d$ be the group of permutations on the set $[d]=\{1,2,...,d\}$.
A permutation $\pi$ can be written as a product of its disjoint cycles, and
the length distribution of the cycles of $\pi$ is called the cycle-type of $\pi$ and denoted by $ct(\pi)$.
One generally writes $ct(\pi)$ as an (integer) partition
of $d$. A partition $\lambda$ of $d$, denoted by $\lambda \vdash d$,
is usually represented by a nonincreasing positive integer sequence $\lambda=(\lambda_1, \lambda_2,\ldots, \lambda_k)$
such that $\lambda_1+\lambda_2+\cdots+\lambda_k=d$. The number $k$ is called the length of $\lambda$, denoted by $\ell(\lambda)$.
If $\lambda \vdash d$, we also write $|\lambda| =d$.

Another common representation of $\lambda$ with $m_i$ of $i$'s is in the form $\lambda=[1^{m_1}, 2^{m_2},\ldots, d^{m_d}]$.
We often discard the entry $i^{m_i}$ in case $m_i=0$ and write the entry $i^1$ simply as $i$.
Obviously, $\sum_{i=1}^d i m_i=d$ and $\ell(\lambda)=\sum_{i=1}^d m_i$.
One more relevant quantity, denoted by $Aut(\lambda)$, is
the number of permutations of the entries of $\lambda$ that
fix $\lambda$. Clearly, $Aut(\lambda)=\prod_{i=1}^d m_i!$.
In this paper,
we will use the two representations of partitions interchangably, whichever is more convenient.

It is well known that a conjugacy class of $\mathfrak{S}_d$ consists of
permutations of the same cycle-type.
So, the conjugacy classes can be indexed by partitions of $d$.
Let $\mathcal{C}_{\lambda}$ denote the one indexed by $\lambda$.
If $\pi \in \mathcal{C}_{\lambda}$, then the number of cycles contained in $\pi$ is $\ell(\pi)=\ell(\lambda)$.
Moreover, the number of elements contained in $\mathcal{C}_{\lambda}$ is well known to be 
$$
|\mathcal{C}_{\lambda}|=\frac{d!}{z_{\lambda}}, \quad \mbox{where $z_{\lambda}= \prod_{i=1}^d i^{m_i} m_i! = Aut(\lambda)\prod_{i=1}^d i^{m_i}  $.}
$$
Recall the generating function of the signless Stirling
numbers of the first kind $\stirlingI{n}{k}$ is given by:
$$
\sum_{k=1}^n (-1)^{n-k}\stirlingI{n}{k} x^k = x(x-1)\cdots (x-n+1).
$$
The number $\stirlingI{n}{k}$ counts permutations on $[n]$ having exactly $k$ cycles.

From the representation theory of the symmetric group $\mathfrak{S}_d$,
the number of irreducible representations is known to be the same as the number of its conjugacy classes.
Consequently, we can index the irreducible representations by the partitions of $d$ as well.
We write the character associated to the irreducible representation indexed by $\lambda$
as $\chi^{\lambda}$ and the dimension of the irreducible representation as $\dim(\lambda)$. Suppose $C$ is a conjugacy class of $\mathfrak{S}_d$ indexed by $\alpha$ and $\pi \in C$.
We will view $\chi^{\lambda}(\pi), \, \chi^{\lambda}(\alpha), \, \chi^{\lambda}(C)$ and $\chi_{\alpha}^{\lambda}$ etc.~as the same, and we trust the context to prevent confusion. 

The following lemmas are well known and will be used later.

\begin{lemma}
	Suppose $\alpha =(d) \vdash d$. Then,
	\begin{align}
	\chi^\lambda (\alpha ) =\begin{cases}
		(-1)^j, & \mbox{if $\lambda=[1^j, d-j]$ for $0 \leq j \leq d-1$}, \\
		0, & \mbox{others}.
	\end{cases}
\end{align}
\end{lemma}
In the above lemma, $[{1^{d - 1}},1] = [{1^d}]$ and other cases of this kind will be treated analogously.

\begin{lemma}
	For $\lambda =[1^j,d-j]$, $\dim(\lambda) ={d-1 \choose j}$.
\end{lemma}

\begin{lemma}[Jackson~\cite{jac}] \label{lem:jac}
Let $\beta=[1^{a_1},2^{a_2},\ldots,d^{a_d}]\vdash d$. Then, we have
$$
\sum_{j=0}^{d-1}(-1)^j\chi^{[1^j,d-j]}(\beta)y^j=(1-y)^{-1}\prod_{j=1}^{d}\left\{1-y^j  \right\}^{a_j}.
$$
\end{lemma}

\subsection{A recursion implying the Frobenius identity}

Let 
\begin{align*}
	\mathfrak{m}_{\lambda,m}= \prod_{u\in \lambda} \frac{m+c(u)}{h(u)}, \qquad 	\mathfrak{c}_{\lambda,m} &= \sum_{k=0}^m (-1)^k {m \choose k} \mathfrak{m}_{\lambda,m-k},
\end{align*}
where for $u=(i,j)\in \lambda$, i.e., the $(i,j)$ cell in the Young diagram of $\lambda$, $c(u)=j-i$ and $h(u)$ is the hook length of 
the cell $u$. 
The following two theorems were proved in~\cite{chr-hur}.

\begin{theorem}[Chen~\cite{chr-hur}] \label{thm:main-recur} 
	Let $\xi_{d,m}(C_1,\ldots, C_t)$ be the number of tuples $( \sigma_1,\sigma_2,\ldots ,\sigma_t)$ 
	such that the permutation $\pi=  \sigma_1\sigma_2\cdots \sigma_t$ has $m$ cycles, where $\sigma_i $
	belongs to a conjugacy class ${C}_i$ of $\mathfrak{S}_d$. Then we have
	\begin{align}\label{eq:main-recur}
		\xi_{d,m}(C_1,\ldots, C_t) = W_{d,m}(C_1,\ldots, C_t)-\sum_{k>m} S(k,m) \xi_{d,k}(C_1,\ldots, C_t),
	\end{align}
	where $S(k, m)$ is the Stirling number of the second kind enumerating partitions of a set with $k$
	elements into $m$ blocks, and
	\begin{align}\label{eq:W}
		W_{d,m}(C_1,\ldots, C_t)= \frac{\prod_{i=1}^{t}|C_i|}{m!} \sum_{\lambda \vdash d} \mathfrak{c}_{\lambda,m}	 
		\frac{1}{\big\{\dim(\lambda) \big\}^{t-1}} \prod_{i=1}^t \chi^{\lambda}(C_i).
	\end{align}
	
\end{theorem}

\begin{theorem}[Chen~\cite{chr-hur}] \label{thm:main-recur11} 
	Let $\xi_{d,m}(C_1,\ldots, C_t)$ be the number of tuples $( \sigma_1,\sigma_2,\ldots ,\sigma_t)$ 
	such that the permutation $\pi=  \sigma_1\sigma_2\cdots \sigma_t$ has $m$ cycles, where $\sigma_i $
	belongs to a conjugacy class ${C}_i$ of $\mathfrak{S}_d$. Then we have
	\begin{align}\label{eq:main-recur11}
		\xi_{d,m}(C_1,\ldots, C_t) 
		=\sum_{k = 0}^{d-m} (-1)^k \stirlingI{m+k}{m}   W_{d,m+k}(C_1,\ldots, C_t).
	\end{align}

\end{theorem}

We remark that when $m=d$, eq.~\eqref{eq:main-recur11}
reduces to the Frobenius identity, see~\cite{chr-hur} for discussion.
Theorem~\ref{thm:main-recur11} follows from iterating eq.~\eqref{eq:main-recur}
and simplifying the resulting coefficients. 

A permutation in $\mathfrak{S}_n$ with only one cycle is called a full cycle or an $n$-cycle,
and a permutation of the cycle-type $[1^j, n-j]$ is called an $(n-j)$-cycle.
The summation in eq.~\eqref{eq:W} can be simplified if one of the conjugacy classes corresponds to full cycles.

\begin{proposition}[Chen~\cite{chr-hur}]
	\begin{align}\label{eq:WW}
		W_{d,m}(\mathcal{C}_{(d)}, C_1,\ldots, C_t)=\frac{(d-1)!\prod_{i=1}^{t}|C_i|}{m!}\sum_{j=0}^{d-1} (-1)^j \frac{{d-1-j \choose d-m} }{{d-1\choose j}^{t-1}} \prod_{i=1}^t \chi^{[1^j, d-j]}(C_i).
	\end{align}
\end{proposition}

\section{Completed cycles}

In this section, we provide necessary definions and notations regarding
shifted symmetric functions and completed cycles.
For more detailed discussions, we refer to~\cite{ok-pa,KO94} and the references therein.

\subsection{Shifted symmetric functions}

Let $\mathbb{Q}[x_1,...,x_d]$ be the algebra of $d$-variable polynomials over $\mathbb Q$. The shifted action of $\mathfrak{S}_d$ on this algebra is defined by:
$$
\sigma (f(x_1-1,\ldots,x_d-d)):= f(x_{\sigma(1)}-\sigma(1),\ldots,x_{\sigma(d)}-\sigma(d))
$$
for any $\sigma \in \mathfrak{S}_d$ and for any polynomial $f$ written in the variables $x_i - i$. Denote by $\mathbb{Q}[x_1,...x_d]^*$ the subalgebra of polynomials that are invariant under this action.

\begin{definition}
The algebra of shifted symmetric functions is
$$
\Lambda^* = \lim_{\longleftarrow} \mathbb{Q}[x_1,\ldots,x_d]^{*}
$$
where the project limit is taken in the category of filtered algebras with respect to the homomorphism which sends the last variable to $0$. 
\end{definition}

Concretely, an element $\textbf{f}\in \Lambda^*$ is a sequence,
$$
\textbf{f} = \left\{ \textbf{f}^{(d)} \right\},\qquad  \textbf{f}^{(d)} \in \mathbb{Q}[x_1,\ldots,x_d]^*
$$
which satisfies the following conditions:
\begin{enumerate}
  \item the polynomials $\textbf{f} ^{(d)}$ are of uniformly bounded degree,
  \item the polynomials $\textbf{f} ^{(d)}$ are stable under restriction, i.e.,
      $$  \textbf{f} ^{(d+1)} \big |_{x_{d+1}=0 } = \textbf{f}^{(d)}.
      $$
\end{enumerate}
We may view $\textbf{f}$ as the limit $\lim_{d \rightarrow \infty} \textbf{f} ^{(d)}$.
We can evaluate any $\textbf{f} \in \Lambda^*$ at any point $(x_1, x_2,\ldots) \in \mathbb{Q}^{\infty}$
whose all but a finite number of entries are zero.
Note that each partition $(\lambda_1, \ldots, \lambda_n)$ induces $(\lambda_1, \ldots, \lambda_n,0,0,\ldots) \in \mathbb{Q}^{\infty}$.
We simply write the evaluation of $\textbf {f}$ at the latter as at the former.
It is well known that $\textbf{f}$ is uniquely determined by their values $\textbf{f}(\lambda)$ for partitions $\lambda \in \textbf{Par}$,
where $\textbf{Par}$ stands for the set of all partitions of integers.

Next, we will introduce two bases for $\Lambda^*$.
The first basis is $\{p_{\lambda}:\lambda \in \textbf{Par} \}$, where for $\lambda=(\lambda_1, \ldots, \lambda_n)$,
$p_{\lambda}= p_{\lambda_1} \cdots p_{\lambda_n}$, and
$$
p_k(x_1,x_2,\ldots):= \sum_{i=1}^{\infty}\big[(x_i-1+ \frac{1}{2})^k-(-i+ \frac{1}{2})^k) \big],
$$
i.e., a counterpart of the usual power sum symmetric function.

An alternative definition of the shifted power sum symmetric function is as follows:
$$
p^*_k(x_1,x_2,\ldots):= \sum_{i=1}^{\infty} \big[ (x_i-1+ \frac{1}{2})^k-(-i+ \frac{1}{2})^k \big]+ (1-2^{-k}) \zeta(-k).
$$
Since $p^*_k$ and $p_k$ only differ by a constant (i.e., the Riemann zeta value part), $p^*_k$ are clearly shifted symmetric and also a basis.
In this paper, most of the time, we will be working with $p_k$.
However, in the GW/H correspondence~\cite{ok-pa}, $p^*_k$ are the ones needed.
In these situations, we will explain how to remedy the gap. 

The second basis is $\{f_{\lambda}:\lambda \in \textbf{Par} \}$, where for $\lambda=(\lambda_1, \ldots, \lambda_n)$,
$$
f_{\lambda}(\mu):= {|\mu| \choose |\lambda|} |\mathcal{C}_{\lambda}| \frac{\chi^{\mu}_{\lambda}}{\dim (\mu)}.
$$
See Kerov and Olshanski~\cite{KO94}, and also Okounkov and Olshanski~\cite{OO98}. 
Here, if $|\mu|< |\lambda|$, the returned value is zero due to the binomial factor.
If $|\mu| > |\lambda|$, the character $\chi_{\lambda}^{\mu}$ is viewed as the evaluation
at an element in the subgroup $\mathfrak{S}_{|\lambda|}$ of $\mathfrak{S}_{|\mu|}$.
If $\lambda = \varnothing$,  the formula is interpreted as the constant function $f_{\varnothing} (\mu)=1$.

\subsection{The preimage of $p_k$ under an isomorphism}

Let $\mathcal{Z}_d$ be the center of the group algebra of $\mathfrak{S}_d$ (over $\mathbb{Q}$).
It is known that the map
\begin{align}
\phi:	\oplus_{d=0}^{\infty} \mathcal{Z}_d \rightarrow \Lambda^{*}, \quad \mathcal{C}_{\lambda} \mapsto f_{\lambda},
\end{align}
is a linear isomorphism.
Now suppose 
\begin{align} \label{eq:pf-trans}
p_{\lambda} = \sum_{\mu} \kappa_{\lambda, \mu} f_{\mu}.
\end{align}
Then, 
\begin{align}
\phi^{-1} (p_{\lambda}) =\sum_{\mu} \kappa_{\lambda, \mu} \mathcal{C}_{\mu}.
\end{align}
We are particularly interested in $\phi^{-1} (p_{\lambda})$ for $\lambda$ containing only one part.
More specifically,

\begin{definition}
 For $\lambda=(k)$, $\frac{1}{k!} \phi^{-1} (p_{\lambda}) $ is called the
completed $k$-cycle, denoted by $\overline{(k)}$.
\end{definition}

We remark that our normalization follows Shadrin, Spitz and Zvonkine~\cite{complete11}, and differ from these in Okounkov and Pandharipande~\cite{ok-pa}
by a factor $\frac{1}{(k-1)!}$.
In terms of $p_k^*$, the corresponding completed $k$-cycle is denoted
by $\overline{(k)^*}$.
Note that the constant $(1-2^{-k}) \zeta(-k)$ is shifted symmetric, and 
$$
(1-2^{-k}) \zeta(-k)= (1-2^{-k}) \zeta(-k) f_{\varnothing}.
$$
Thus, by construction,
$$
\phi^{-1} ((1-2^{-k}) \zeta(-k))= (1-2^{-k}) \zeta(-k) \mathcal{C}_{\varnothing}.
$$
That is, 
\begin{align}
\overline{(k)^*} =\overline{(k)} + \frac{ (1-2^{-k}) \zeta(-k)}{k!} \mathcal{C}_{\varnothing}.
\end{align}

The formulas for the first few completed cycles are:
$$
\begin{aligned}
& 0!\cdot \overline{(1)}=\mathcal{C}_{(1)},\\
& 1!\cdot \overline{(2)}=\mathcal{C}_{(2)},\\
& 2!\cdot \overline{(3)}=\mathcal{C}_{(3)}+\mathcal{C}_{(1,1)}+\tfrac{1}{12}\cdot \mathcal{C}_{(1)},\\
& 3!\cdot \overline{(4)}=\mathcal{C}_{(4)}+2\cdot \mathcal{C}_{(2,1)}+\tfrac{5}{4}\cdot \mathcal{C}_{(2)}.
\end{aligned}
$$

The coefficients $\kappa_{\lambda, \mu}$ were completely determined in Okounkov and Pandharipande~\cite{ok-pa}
for $\lambda=(k)$. In particular, if $|\mu|>k$, then $\kappa_{(k), \mu}=0$.
Moreoever, $\kappa_{(k), \mu} \geq 0$.

\section{Explicit formulas}

In this section, we first define Hurwitz numbers with completed cycles,
and then devote our effort to computing certain so-called triple
Hurwitz numbers with completed cycles.

\subsection{Definition of Hurwitz numbers with completed cycles}

We first give a combinatorial definition of Hurwitz numbers with completed cycles.
Suppose for $1\leq i \leq r$, $\mu^i \vdash d$.
The $r$-fold Hurwitz number $H_d(k_1,\ldots, k_s; \mu^1, \ldots, \mu^r)$ with completed $k_1, \ldots, k_s$-cycles
is defined as a certain weighted sum of tuples of permutations as follows:
\begin{align}
H_d(k_1,\ldots, k_s; \mu^1, \ldots, \mu^r) = \sum_{\eta_i \in \textbf{Par}, \, |\eta_i| \leq d, \atop 1\leq i \leq s} \frac{1}{d!}\prod_{i=1}^s \frac{ \kappa_{k_i, \eta_i} }{k_i !} \sum_{\sigma_1 \cdots \sigma_s \pi_1 \cdots \pi_r =1, \atop ct(\sigma_i)= \eta_i^{\uparrow d}, \, ct(\pi_j) = \mu^j } 1 ,
\end{align}
where for $|\eta_i| \leq d$,  $\eta_i^{\uparrow d}$ is a partition of $d$ obtained from adding $d-|\eta_i|$ of $1$'s to $\eta_i$.

The counterpart with $\overline{(k)}$ replaced with $\overline{(k)^*}$ is denoted
by $H_d^* (k_1,\ldots, k_s; \mu^1, \ldots, \mu^r)$.
In view of the difference between $\overline{(k)}$ and $\overline{(k)^*}$, it is not difficult to show
\begin{align}
H_d^* (k_1,\ldots, k_s; \mu^1, \ldots, \mu^r)= \sum_{1\leq i_1 < \cdots < i_t \leq s, \atop t\geq 0} \prod_{j=1}^t \frac{(1-2^{-k_{i_j}} )\zeta(- k_{i_j})}{k_{i_j} !} H_d(k_1',\ldots, k_{s-t}'; \mu^1, \ldots, \mu^r),
\end{align}
where the formed sequence $(k_1', \ldots, k_{s-t}')$ is the one resulted from removing $k_{i_j}$
from the sequence $(k_1, \ldots, k_s)$.
Therefore, computing $H_d$ and $H_d^*$ are essentially the same, and we will be focusing on the former hereafter. 

\begin{theorem}[Frobenius identity]\label{thm:fro}
	Let $C_i$ be a conjugacy class of $\mathfrak{S}_d$, $1\leq i \leq k$, and $N_{C_1,C_2,\ldots, C_k}$ denote the number of tuples $(\sigma_1, \sigma_2,\ldots, \sigma_k)$
	such that $\sigma_i \in C_i$ and $\sigma_1 \sigma_2 \cdots \sigma_k=1$.
	Then,
	\begin{align}
		N_{C_1,C_2,\ldots, C_k}=\sum_{\lambda \vdash d} \frac{\prod_{i=1}^{k}|C_i|  \chi^{\lambda}(C_i)}{d!}  \{ \dim(\lambda) \}^{-k+2} .
	\end{align}
\end{theorem}

It may be worth pointing out that the Frobenius identity is sometimes attributed to Burnside.

\begin{lemma}
The following formulas hold:
	\begin{align}
		H_d(k_1,\ldots, k_s; \mu^1, \ldots, \mu^r) &=\sum_{\lambda \vdash d} \left(\frac{\dim(\lambda)}{d!} \right) ^2 \left\{\prod_{i=1}^{s}\frac{p_{k_i}(\lambda)}{k_i!}\right\}\left\{\prod_{i=1}^{r}\frac{|\mathcal{C}_{\mu^i}|\chi^{\lambda}_{\mu^i}}{\dim(\lambda)} \right\} ,\\
		H_d^*(k_1,\ldots, k_s; \mu^1, \ldots, \mu^r) &=\sum_{\lambda \vdash d} \left(\frac{\dim(\lambda)}{d!} \right) ^2 \left\{\prod_{i=1}^{s}\frac{p_{k_i}^*(\lambda)}{k_i!}\right\}\left\{\prod_{i=1}^{r}\frac{|\mathcal{C}_{\mu^i}|\chi^{\lambda}_{\mu^i}}{\dim(\lambda)} \right\} .
	\end{align}
\end{lemma}

\proof By definition, we have
\begin{align*}
	& \quad H_d(k_1,\ldots, k_s; \mu^1, \ldots, \mu^r) \\
	= &  \sum_{\eta_i \in \textbf{Par}, \, |\eta_i| \leq d, \atop 1\leq i \leq s} \frac{1}{d!}\prod_{i=1}^s  \frac{\kappa_{k_i, \eta_i} }{k_i !} \sum_{\sigma_1 \cdots \sigma_s \pi_1 \cdots \pi_r =1, \atop ct(\sigma_i)= \eta_i^{\uparrow d}, \, ct(\pi_j) = \mu^j } 1 \\
		= &  \sum_{\eta_i \in \textbf{Par}, \, |\eta_i| \leq d, \atop 1\leq i \leq s} \frac{1}{d!}\prod_{i=1}^s \frac{ \kappa_{k_i, \eta_i}}{k_i !}  N_{\mathcal{C}_{\eta_1^{\uparrow d}}, \ldots, \mathcal{C}_{\eta_s^{\uparrow d}}, \mathcal{C}_{\mu^1}, \ldots, \mathcal{C}_{\mu^r}} \\
	=& \sum_{\eta_i \in \textbf{Par}, \, |\eta_i| \leq d, \atop 1\leq i \leq s} \frac{1}{d!}\prod_{i=1}^s \frac{\kappa_{k_i, \eta_i}}{k_i !}  \sum_{\lambda \vdash d} \frac{\dim(\lambda)^2}{d!} \prod_{i=1}^s \frac{ |\mathcal{C}_{\eta_i^{\uparrow d}}| \chi^{\lambda} (\eta_i^{\uparrow d})}{\dim(\lambda)} \prod_{i=1}^r \frac{ |\mathcal{C}_{\mu^i}| \chi^{\lambda} (\mu^i)}{\dim(\lambda)} .
\end{align*}
It is not hard to verify that
$$
 |\mathcal{C}_{\eta_i^{\uparrow d}}| = {|\eta_i^{\uparrow d}| \choose |\eta_i|} |\mathcal{C}_{\eta_i}|,
$$
and, consequently, the last formula equals
\begin{align*}
& \sum_{\eta_i \in \textbf{Par}, \, |\eta_i| \leq d, \atop 1\leq i \leq s} \frac{1}{d!}\prod_{i=1}^s \frac{\kappa_{k_i, \eta_i} }{k_i !}  \sum_{\lambda \vdash d} \frac{\dim(\lambda)^2}{d!} \prod_{i=1}^s {|\eta_i^{\uparrow d}| \choose |\eta_i|} \frac{ |\mathcal{C}_{\eta_i}| \chi^{\lambda} (\eta_i)}{\dim(\lambda)} \prod_{i=1}^r \frac{ |\mathcal{C}_{\mu^i}| \chi^{\lambda} (\mu^i)}{\dim(\lambda)} \\
=& \sum_{\lambda \vdash d} \frac{\dim(\lambda)^2}{ (d!)^2} \prod_{i=1}^r \frac{ |\mathcal{C}_{\mu^i}| \chi^{\lambda} (\mu^i)}{\dim(\lambda)} 
\prod_{i=1}^s \sum_{\eta_i \in \textbf{Par}, \, |\eta_i| \leq d}  \frac{\kappa_{k_i, \eta_i} }{k_i !}  {|\eta_i^{\uparrow d}| \choose |\eta_i|} \frac{ |\mathcal{C}_{\eta_i}| \chi^{\lambda} (\eta_i)}{\dim(\lambda)} \\
=& \sum_{\lambda \vdash d} \frac{\dim(\lambda)^2}{ (d!)^2} \prod_{i=1}^r \frac{ |\mathcal{C}_{\mu^i}| \chi^{\lambda} (\mu^i)}{\dim(\lambda)} 
\prod_{i=1}^s \sum_{\eta_i \in \textbf{Par}, \, |\eta_i| \leq d} \frac{ \kappa_{k_i, \eta_i}}{k_i !}  f_{\eta_i} (\lambda) \\
=& \sum_{\lambda \vdash d} \frac{\dim(\lambda)^2}{ (d!)^2} \prod_{i=1}^r \frac{ |\mathcal{C}_{\mu^i}| \chi^{\lambda} (\mu^i)}{\dim(\lambda)} 
\prod_{i=1}^s \frac{p_{k_i}(\lambda)}{k_i !}.
\end{align*}
The other equation is analogous, and the proof follows. \qed

The equations in the above lemma may be known to people, but we have not
found a detailed derivation of them. So, we include a proof here for completeness.
Also, we will not directly use the equations in this paper.

The following remarkable correspondence between the stationary sector
of the Gromov--Witten theory of $X=\mathbb{CP}^1$ and Hurwitz numbers with completed cycles has
been established~\cite{ok-pa}.

\begin{theorem}[GW/H correspondence~\cite{ok-pa}]
	There holds
	\begin{align}\label{eq:gw-h}
		\left\langle \prod_{i=1}^{s}\tau_{k_i}(\omega),\mu^1, \ldots, \mu^r \right \rangle^{X}_{g,d} = H_d^*(k_1+1,\ldots, k_s+1; \mu^1, \ldots, \mu^r).
	\end{align}
\end{theorem}

Note that $	\left\langle \prod_{i=1}^{s}\tau_{k_i}(\omega),\mu^1, \ldots, \mu^r \right \rangle^{X}_{g,d}$ vanishes
unless the dimension constraint is satisfied:
\begin{align}\label{eq:dimen-constraint}
	2g-2+2d= \sum_{i=1}^s k_i + \sum_{j=1}^r [d-\ell(\mu^i)].
\end{align}
In view of the GW/H correspondence, we can compute the stationary sector of the Gromov--Witten
theory by calculating Hurwitz numbers, which partly motivated our study of the latter.

Let 
\begin{align}
	H_{d,m}(k_1,\ldots, k_s; \mu^1, \ldots, \mu^r) := \sum_{\ell(\mu^{r+1})=m} H_d(k_1,\ldots, k_s; \mu^1, \ldots, \mu^r, \mu^{r+1}).
\end{align}
We call these numbers quasi-$(r+1)$-fold Hurwitz numbers with completed cycles.
Then, Theorem~\ref{thm:main-recur} implies the following theorem due to linearity.

\begin{theorem}
	We have
	\begin{align}\label{eq:main-recur}
		H_{d,m}(k_1,\ldots, k_s; \mu^1, \ldots, \mu^r) &= \overline{W}_{d,m} (k_1,\ldots, k_s; \mu^1, \ldots, \mu^r) \nonumber \\
		& \qquad -\sum_{k>m} S(k,m) H_{d,k}(k_1,\ldots, k_s; \mu^1, \ldots, \mu^r),
	\end{align}
	where
	\begin{align}
		\overline{W}_{d,m}
		=& \frac{1}{d! m!} \sum_{\lambda \vdash d} \mathfrak{c}_{\lambda,m} \frac{\prod_{i=1}^r |\mathcal{C}_{\mu^i}| \chi^{\lambda}(\mathcal{C}_{\mu^i})}{\dim(\lambda)^{r-1}} \prod_{i=1}^s \frac{p_{k_i} (\lambda)}{k_i!}.
	\end{align}
	
\end{theorem}

\proof Note that using the notation in Theorem~\ref{thm:main-recur}, we first have
\begin{align*}
	H_{d,m}(k_1,\ldots, k_s; \mu^1, \ldots, \mu^r) = &  \sum_{\eta_i \in \textbf{Par}, \, |\eta_i| \leq d, \atop 1\leq i \leq s} \frac{1}{d!}\prod_{i=1}^s \frac{ \kappa_{k_i, \eta_i}}{k_i !}  \xi_{d,m}(\mathcal{C}_{\eta_1^{\uparrow d}}, \ldots, \mathcal{C}_{\eta_s^{\uparrow d}}, \mathcal{C}_{\mu^1}, \ldots, \mathcal{C}_{\mu^r})  .
\end{align*}
Applying Theorem~\ref{thm:main-recur}, we then obtain
$$
H_{d,m}(k_1,\ldots, k_s; \mu^1, \ldots, \mu^r) = \overline{W}_{d,m} -\sum_{k>m} S(k,m) H_{d,k}(k_1,\ldots, k_s; \mu^1, \ldots, \mu^r),
$$
where
\begin{align*}
	&\overline{W}_{d,m}=  \sum_{\eta_i \in \textbf{Par}, \, |\eta_i| \leq d, \atop 1\leq i \leq s} \frac{1}{d!}\prod_{i=1}^s \frac{ \kappa_{k_i, \eta_i}}{k_i !}  W_{d,m}\\
	=& \sum_{\eta_i \in \textbf{Par}, \, |\eta_i| \leq d, \atop 1\leq i \leq s} \frac{1}{d!}\prod_{i=1}^s \frac{ \kappa_{k_i, \eta_i}}{k_i !}  \frac{\prod_{i=1}^{s}|\mathcal{C}_{\eta_i^{\uparrow d}}| \cdot \prod_{i=1}^{r}|\mathcal{C}_{\mu^i}|}{m!} \sum_{\lambda \vdash d} \mathfrak{c}_{\lambda,m}	 
	\frac{\prod_{i=1}^s \chi^{\lambda}(\mathcal{C}_{\eta_i})  \prod_{i=1}^r \chi^{\lambda}(\mathcal{C}_{\mu^i})}{\big\{\dim(\lambda)\big\}^{r+s-1}} \\
	=& \frac{1}{d! m!} \sum_{\lambda \vdash d} \mathfrak{c}_{\lambda,m} \frac{\prod_{i=1}^r |\mathcal{C}_{\mu^i}| \chi^{\lambda}(\mathcal{C}_{\mu^i})}{\dim(\lambda)^{r-1}} \prod_{i=1}^s \frac{p_{k_i} (\lambda)}{k_i!},
\end{align*}
completing the proof. \qed

Next, by virtue of Theorem 2.3, we have

\begin{theorem}
	\begin{align}
H_{d,m}(k_1,\ldots, k_s; \mu^1, \ldots, \mu^r) = \sum_{k = 0}^{d-m} (-1)^k \stirlingI{m+k}{m} \overline{W}_{d,m+k} (k_1,\ldots, k_s; \mu^1, \ldots, \mu^r).
	\end{align}
\end{theorem}

In view of Theorem 4.3 and 4.4, the computation of $H$-numbers essentially comes down to that of $\overline{W}$-numbers first.
In the rest of the paper, our subject will be the case $r=2$.
More general cases will be left for future investigation.

\subsection{One-part triple Hurwitz numbers}

In this section, we compute $H_{d,m}(k_1,\ldots,k_s;(d),\beta)$ which will be called one-part quasi-triple Hurwitz
numbers with completed cycles.

It is clear that $\overline{W}_{d,m}(k_1,\ldots,k_s;(d),\beta)$ and $H_{d,m}(k_1,\ldots,k_s;(d),\beta)$ do not
depend on the order of the numbers $k_1,\ldots, k_s$. Without loss of generality, we may assume $k_1 \geq k_2 \geq \cdots \geq k_s$.
Throughout the rest of the paper, if otherwise explicitly stated, we assume
$$
\begin{aligned}
&K=(k_1, k_2, \ldots, k_s)=[1^{b_1},2^{b_2},\ldots,l^{b_l}] \vdash d'\\ 
&\beta=(\beta_1, \beta_2, \ldots, \beta_n)=[1^{a_1},2^{a_2},\ldots,d^{a_d}] \vdash d.
\end{aligned}
$$
Moreover, we may write $\overline{W}_{d,m}(k_1,\ldots,k_s;(d),\beta)$ simply as $\overline{W}_{d,m}(k_1,\ldots,k_s;\beta)$, or even simpler as $\overline{W}_{d,m}(K;\beta)$,
and write $H_{d,m}(k_1,\ldots,k_s;(d),\beta)$ as $H_{d,m}(k_1,\ldots,k_s;\beta)$ or $H_{d,m}(K;\beta)$.
As usual, 
$$
[x_1^{k_1} \cdots x_m^{k_m}] f(x_1,\ldots, x_m)
$$
 stands for the coefficient of the term $x_1^{k_1} \cdots x_m^{k_m}$
in $f(x_1,\ldots, x_m)$.

\begin{lemma}\label{lem:K-bi}
We have the following formula:
$$
	\begin{aligned}
		&\quad  \prod_{i=1}^{s}\frac{p_{k_i}([1^k, d-k])}{k_i!}  = 
		\prod_{i=1}^{l} \frac{b_i!}{(i!)^{b_i}}\left[t_i^{b_i}\right]\prod_{j=0}^{i-1}\sum_{h_i=0}^{\infty}\frac{t_i^{h_i}}{h_i!} \binom{i}{j}^{h_i}  d^{(i-j)h_i}(-k-1/2)^{jh_i}. \\
	\end{aligned}
	$$
\end{lemma}
\begin{proof}
From the definition of $p_i$, we  can derive:
$$
p_i([1^k, d-k])=(d-k-1/2)^i-(-k-1/2)^i,
$$
and
$$
\begin{aligned}
\left( \frac{p_i(d-k,1,..,1)}{i!}\right)^{b_i}
& =
 \left(\frac{(d-k-1/2)^i-(-k-1/2)^i}{i!}\right)^{b_i} \\
 & =
  \frac{b_i!}{(i!)^{b_i}}\left[t_i^{b_i}\right]\exp\left((d-k-1/2)^it_i-(-k-1/2)^it_i\right) \\
 & =
  \frac{b_i!}{(i!)^{b_i}}\left[t_i^{b_i}\right]\exp\left(t_i\sum_{j=0}^{i-1} \binom{i}{j}  d^{i-j}(-k-1/2)^j \right) \\
  &=\frac{b_i!}{(i!)^{b_i}}\left[t_i^{b_i}\right]\prod_{j=0}^{i-1}\sum_{h_i=0}^{\infty}\frac{t_i^{h_i}}{h_i!} \binom{i}{j}^{h_i}  d^{(i-j)h_i}(-k-1/2)^{jh_i}.
\end{aligned}
$$
The rest is clear and the lemma follows. 
\end{proof}

For convenience, we introduce the following two functions:
\begin{align*}
\mathcal{U}(t_1,\ldots, t_l; z) & := e^{\sum_{i=1}^l  \{(1+z)^i - z^i\} t_i},\\
\mathcal{V}_{\beta}(x,y) & := \prod_{v=1}^{d}\left\{ x^v-e^{-yv} \right\}^{c_v}, \quad \mbox{for $\beta= [1^{c_1+1},2^{c_2},\ldots, d^{c_d}] \vdash d $.}
\end{align*}
Now we are in a position to present explicit formulas for one-part quasi-triple Hurwitz numbers with completed cycles.

\begin{theorem}[Explicit formula]\label{thm:explicit-formula}
	Suppose, in the tuple $K=(k_1,\ldots, k_s)$, the largest number is $l$, and there are $b_i$ of $i$'s for $1\leq i \leq l$, and $\beta=[1^{a_1}, \ldots, d^{a_d}] \vdash d$. Then, we have 
\begin{align}
H_{d,m} (k_1,\ldots,k_s;(d),\beta) =\sum_{i=0}^{d-m} (-1)^i \stirlingI{m+i}{m } \overline{W}_{d,m+i}(k_1,\ldots,k_s;(d),\beta)
\end{align}
with
	\begin{equation} \label{eq:triple-comp}
		\begin{split}
		 \overline{W}_{d,q}(k_1,\ldots,k_s;(d),\beta)
			& = 
			C^{q}_{K,\beta} \sum_{t =0}^{d'-s}  [t_1^{b_1}\cdots t_l^{b_l} z^t y^t] \, \frac{t!}{d^t} \mathcal{U}(t_1,\ldots, t_l; z) \\
&  \quad \times \frac{\partial^{d-q}}{\partial  x^{d-q}} e^{-y/2}
			\mathcal{V}_{\beta}(x,y) \bigg |_{x=1}  ,
		\end{split}
	\end{equation}
	where
	$$
	\begin{aligned}
		C^{q}_{K,\beta} &=\frac{|\mathcal{C}_{\beta}| d^{d'-1}}{q!(d-q)! }
		\prod_{i=1}^{l} \frac{b_i!}{(i!)^{b_i}}.
	\end{aligned}
	$$
\end{theorem}

\begin{proof}
	
	First, for $\lambda=[1^k,d-k]$, it is not difficult to derive:
	$$
	\mathfrak{c}_{\lambda,m}=\binom{d-1}{m-1}\binom{m-1}{k}.
	$$
	Then, we obtain	
	$$
	\begin{aligned}
		& \quad \overline{W}_{d,m}(K;\beta) 
		  = 
		\frac{1}{m!d!}\sum_{\lambda \vdash d}\left(\prod_{i=1}^{s} \frac{p_{k_i}(\lambda)}{k_i!}\right) \frac{|\mathcal{C}_{(d)}|\chi_{(d)}^
			{\lambda}|\mathcal{C}_{\beta}|\chi_{\beta}^{\lambda}}{\dim(\lambda)}\mathfrak{c}_{\lambda,m} \\
		& = 
		\frac{(d-1)!|\mathcal{C}_{\beta}|}{m!d!}
		\sum_{k=0}^{d-1}\left(\prod_{i=1}^{s} \frac{p_{k_i}[1^k,d-k]}{k_i!}\right)
		(-1)^k\chi_{\beta}^{[1^k,d-k]}\frac{k!(d-1-k)!}{(d-1)!}\binom{d-1}{m-1}\binom{m-1}{k} \\
		& = 
		\frac{(m-1)!|\mathcal{C}_{\beta}|\binom{d-1}{m-1}}{m!d!}
		\sum_{k=0}^{d-1}\left(\prod_{i=1}^{s} \frac{p_{k_i}[1^k,d-k]}{k_i!}\right)
		(-1)^k\chi_{\beta}^{[1^k,d-k]}(d-k-1)_{d-m}.  
	\end{aligned}
	$$
	Using Lemma~\ref{lem:K-bi}
	$$
	\begin{aligned}
		&\quad  \prod_{i=1}^{s}\frac{p_{k_i}([1^k, d-k])}{k_i!}  = 
		\prod_{i=1}^{l} \frac{b_i!}{(i!)^{b_i}}\left[t_i^{b_i}\right]\prod_{j=0}^{i-1}\sum_{h_i=0}^{\infty}\frac{t_i^{h_i}}{h_i!} \binom{i}{j}^{h_i}  d^{(i-j)h_i}(-k-1/2)^{jh_i},
	\end{aligned}
	$$
we next have
		$$
	\begin{aligned}
	& \quad \overline{W}_{d,m}(K;\beta) 
		 = 
		\frac{(m-1)!|\mathcal{C}_{\beta}|\binom{d-1}{m-1}}{m!d!}
		\sum_{k=0}^{d-1}(-1)^k\chi_{\beta}^{[1^k,d-k]}(d-k-1)_{d-m} \\  
		& \qquad \times \qquad
		\left(\prod_{i=1}^{l} \frac{b_i!}{(i!)^{b_i}}\left[t_i^{b_i}\right]\prod_{j=0}^{i-1}\sum_{h_i=0}^{\infty}\frac{t_i^{h_i}}{h_i!} \binom{i}{j}^{h_i}  d^{(i-j)h_i}(-k-1/2)^{jh_i}\right) \\
		& = 
		\frac{(m-1)!|\mathcal{C}_{\beta}|\binom{d-1}{m-1}}{m!d!}
		\sum_{k=0}^{d-1}(-1)^k\chi_{\beta}^{[1^k,d-k]}(d-k-1)_{d-m} \prod_{i=1}^{l} \frac{b_i!}{(i!)^{b_i}}\left[t_1^{b_1}\cdots t_{l}^{b_l}\right] \\  
		& \qquad \times \qquad
		\Bigg\{  \sum_{h_{i,j} \geq 0 \atop 1 \leq i \leq l, 0 \leq j \leq i-1 } \frac{ \prod_{i}  t_{i}^{\sum_{j} h_{i,j}}}{\prod_{i,j}  (h_{i,j}!)} \left\{ \prod_{i,j} \binom{i}{j}^{h_{i,j}}\right\}  d^{\sum_{i,j} (i-j)h_{i,j}}(-k-1/2)^{\sum_{i,j}  jh_{i,j}}\Bigg\} \\
		& = 
		\frac{(m-1)!|\mathcal{C}_{\beta}|\binom{d-1}{m-1}}{m!d!}
		\prod_{i=1}^{l} \frac{b_i!}{(i!)^{b_i}}\left[t_1^{b_1}\cdots t_{l}^{b_l}\right] 
		\sum_{h_{i,j} \geq 0 \atop 1 \leq i \leq l, 0 \leq j \leq i-1 } \frac{ \prod_{i}  t_{i}^{\sum_{j} h_{i,j}}}{\prod_{i,j}  (h_{i,j}!)} \left\{ \prod_{i,j} \binom{i}{j}^{h_{i,j}}\right\}  \\  
		& \qquad \times \qquad
		d^{\sum_{i,j} (i-j)h_{i,j}} \sum_{k=0}^{d-1}(-1)^k\chi_{\beta}^{[1^k,d-k]}(d-k-1)_{d-m} (-k-1/2)^{\sum_{i,j}  jh_{i,j}} .
	\end{aligned}
	$$
Finally, we have	
	\begin{align}
		\Omega & = 
		\sum_{k=0}^{d-1}(-1)^k\chi_{\beta}^{[1^k,d-k]}(d-k-1)_{d-m} (-k-1/2)^{\sum_{i,j}  jh_{i,j}} \label{O} \\
		& = 
		\sum_{k=0}^{d-1}(-1)^k\chi_{\beta}^{[1^k,d-k]}\left(\frac{\mathrm d^{d-m}}{\mathrm d  x^{d-m}}x^{d-k-1} \bigg |_{x=1} \right)  \left[ \frac{y^{\sum_{i,j}  jh_{i,j}}}{(\sum_{i,j}  jh_{i,j})!}\right] e^{y(-k-1/2)} \notag\\
		& = 
		\frac{\mathrm d^{d-m}}{\mathrm d  x^{d-m}} 
		\left[\frac{y^{\sum_{i,j}  jh_{i,j}}}{(\sum_{i,j}  jh_{i,j})!}\right]  x^{d-1} e^{-y/2}
		\sum_{k=0}^{d-1}(-1)^k\chi_{\beta}^{[1^k,d-k]}x^{-k} e^{-yk} \bigg |_{x=1} \notag\\
		& = 
		\frac{\mathrm d^{d-m}}{\mathrm d  x^{d-m}} 
		\left[\frac{y^{\sum_{i,j}  jh_{i,j}}}{(\sum_{i,j}  jh_{i,j})!}\right]  x^{d-1} e^{-y/2}
		(1-x^{-1}e^{-y})^{-1}\prod_{v=1}^{d}\left\{ 1-(x^{-1}e^{-y})^v \right\}^{a_v}
		\bigg |_{x=1} \notag\\
		& = 
		\frac{\mathrm d^{d-m}}{\mathrm d  x^{d-m}} 
		\left[\frac{y^{\sum_{i,j}  jh_{i,j}}}{(\sum_{i,j}  jh_{i,j})!}\right]  
		e^{-y/2}
		(x-e^{-y})^{-1}\prod_{v=1}^{d}\left\{ x^v-e^{-yv} \right\}^{a_v}
		\bigg |_{x=1} .\notag
	\end{align}
In the above derivation, Lemma~\ref{lem:jac} has been used.
Plugging the expression of $\Omega$ into $\overline{W}_{d,m}(K;\beta) $ and collecting the coefficient of the term $y^t$,
we then observe that besides the contribution from $\Omega$, the contribution from the remaining part is
\begin{align*}
& \frac{|\mathcal{C}_{\beta}|}{m!(d-m)!d}
		\prod_{i=1}^{l} \frac{b_i!}{(i!)^{b_i}} 
		\sum_{h_{i,j} \geq 0,\, 1 \leq i \leq l\atop h_{i,0}+\cdots+h_{i,i-1}=b_i, \, \sum_{i,j} j h_{ij}=t } \frac{ (\sum_{i,j} j h_{ij})!}{\prod_{i,j}  (h_{i,j}!)} \left\{ \prod_{i,j} \binom{i}{j}^{h_{i,j}}\right\} d^{\sum_{i,j} (i-j)h_{i,j}} \\
= & \frac{|\mathcal{C}_{\beta}| d^{\sum_i i b_i}}{m!(d-m)!d}
		\prod_{i=1}^{l} \frac{b_i!}{(i!)^{b_i}} [t_1^{b_1}\cdots t_l^{b_l} z^t] \, \frac{t!}{d^t} e^{\sum_{i=1}^l  \{(1+z)^i - z^i\} t_i} .
\end{align*} 
The rest is clear and the proof follows.
\end{proof}

We remark that the sum over $t\geq 0$ actually involves a finite number of terms depending on $K$.
We may also formulate Theorem~\ref{thm:explicit-formula} into a generating function tracking the ``order" of
the involved completed cycles as follows.

\begin{theorem}[Generating function]\label{thm:gen-func}
	Suppose $\beta=[1^{a_1}, \ldots, d^{a_d}] \vdash d$. Then, we have 
\begin{align}
& \sum_{K=[1^{b_1}, \ldots, l^{b_l}], \atop b_i \geq 0}  H_{d,m} (K;(d),\beta) \frac {(1! t_1)^{b_1} }{d^{1\cdot b_1} b_1!} \cdots \frac {(l! t_l)^{b_l} }{d^{l\cdot b_l} b_l!} =\frac{|\mathcal{C}_{\beta}| }{d!} \sum_{t\geq 0} [ z^t]  \frac{t!}{d^{t+1}}  \mathcal{U}(t_1,\ldots, t_l;z) \nonumber\\
& \quad \qquad \qquad \times [y^t] \sum_{i=0}^{d-m}  (-1)^i \stirlingI{m+i}{m } {d \choose m+i} \frac{\partial^{d-m-i}}{\partial  x^{d-m-i}} e^{-y/2}
			\mathcal{V}_{\beta}(x, y) \bigg |_{x=1}.
\end{align}

\end{theorem}

We next present an alternative expression in terms of the Bernoulli polynomials.
The Bernoulli polynomials of order $N$, ${\bf B}_n^{(N)}(z)$, is defined by the generating function as follows:
$$
B^{(N)}(x,z)=\left(\frac{x}{e^x-1}\right)^N e^{xz}=\sum_{n\geq 0} {\bf B}_n^{(N)}(z) \frac{x^n}{n!}.
$$
The number $B_n^{(N)}:={\bf B}_n^{(N)}(0)$ is called the $n$-th Bernoulli number
of order $N$.
The case $N=1$ gives the classical Bernoulli polynomials $B_n(z)$ and Bernoulli numbers $B_n$.
The following identity holds (see e.g. Srivastava and Todorov~\cite{bernoulli}):
\begin{align}
	{\bf B}_n^{(N)}(x) & =\sum_{k=0}^n {n\choose k} B_k^{(N)} x^{n-k} =\sum_{k=0}^n {n\choose k} B_{n-k}^{(N)} x^{k}. \label{eq:ber-expan}
\end{align}

\begin{theorem}[Bernoulli expression]

$$
	\begin{aligned}
		& \quad \overline{W}_{d,m}(k_1,\ldots,k_s;(d),\beta) \times d \times Aut(\beta)
		 = 
		\frac{(d-m)! d! d^{d'}}{m! \prod_i \beta_i}
		\prod_{i=1}^{l} \frac{b_i!}{(i!)^{b_i}}\\
		&  \quad  \times\qquad \sum_{t \geq 0} [t_1^{b_1}\cdots t_l^{b_l} z^t ] \, \frac{t!}{d^t} \mathcal{U}(t_1,\ldots, t_l; z) \\  
		& \quad  \times\qquad
		\sum_{r=0}^{d-m} \sum_{1 \leq k_1<\cdots<k_p\leq n  \atop p\geq 0}\frac{(-1)^{n-p+1+t}}{(r+1+t)!} 
		\binom{\beta_{k_1}+\cdots+\beta_{k_p}}{d-m-r} {\bf B}_{r+1+t}^{(r+1)}\left( \frac{1}{2} + \sum_{q\notin\{ k_1,\ldots,k_p\} }\beta_{q}\right),
	\end{aligned}
	$$
\end{theorem}

\begin{proof}
The proof is analogous to the Bernoulli expression for standard one-part quasi-triple
Hurwitz numbers~\cite{chr-hur}.
\end{proof}

When $k_i=2$ for all $i$, we obtain the following explicit formula
for standard one-part quasi-triple Hurwitz numbers which is equivalent to the one
first obtained in Chen~\cite{chr-hur}.

\begin{corollary}[Standard one-part quasi-triple Hurwitz numbers]
	 The standard one-part quasi-triple Hurwitz numbers are given by
	 $$
	 H_{d,m}(2,\ldots,2;(d),\beta)= \sum_{k=0}^{d-m} (-1)^k \stirlingI{m+k}{m} \overline{W}_{d,m}(2,\ldots,2;(d),\beta),
	 $$
where there are $s$ of $2$'s, and
$$
\begin{aligned}
	 \quad \overline{W}_{d,m}(2,\ldots,2;(d),\beta)
	 = 
	\frac{|\mathcal{C}_{\beta}| d^{s-1}}{m!(d-m)!}
	\frac{\mathrm d^{d-m}}{\mathrm d  x^{d-m}} 
	\left[y^{s}\right]  
	e^{\frac{(d-1)y}{2}}
	\mathcal{V}_{\beta}(x, y)
	\bigg |_{x=1}.
\end{aligned}
$$
\end{corollary}

\proof Putting $k_i=2$ in eq.~\eqref{eq:triple-comp}, we immediately obtain
\begin{align*}
&\quad  \overline{W}_{d,m}(2,\ldots,2;(d),\beta)\\
& = 
\frac{|\mathcal{C}_{\beta}|}{m!d(d-m)!}
\frac{1}{2^{s}} \sum_{h = 0}^s   2^{s-h}  d^{s+h} 
\frac{\mathrm d^{d-m}}{\mathrm d  x^{d-m}} 
\left[y^{s-h}\right]  
e^{-y/2}
(x-e^{-y})^{-1}\prod_{v=1}^{d}\left\{ x^v-e^{-yv} \right\}^{a_v}
\bigg |_{x=1}\\
& = 
\frac{|\mathcal{C}_{\beta}| d^{s-1}}{m!(d-m)!}
\frac{\mathrm d^{d-m}}{\mathrm d  x^{d-m}} 
\left[y^{s}\right]  
e^{\frac{(d-1)y}{2}}
(x-e^{-y})^{-1}\prod_{v=1}^{d}\left\{ x^v-e^{-yv} \right\}^{a_v}
\bigg |_{x=1}. 
\end{align*}
The rest is easy to complete. \qed

\subsection{One-part double Hurwitz numbers}

Recall the hyperbolic sin function $\sinh(x)=(e^x-e^{-x})/2$.
For $j>0$,
let 
\begin{align*}
\xi_{2j} &=[x^{2j}]\log(\sinh x/x) ,\\
S_{2j} &=-1+\sum_{k\geq 1}\beta_k^{2j} .
\end{align*}
For $\lambda=(\lambda_1,\lambda_2,\ldots)$, let $\xi_{\lambda}= \xi_{\lambda_1}\xi_{\lambda_2}\cdots$ and $S_{\lambda}=S_{\lambda_1}S_{\lambda_2}\cdots$ and $2\lambda=(2\lambda_1,2\lambda_2,\ldots)$. Then, we have the following identity which can be found in Jackson~\cite{jac}:
	\begin{equation}
    \begin{split}
	\prod_{v\geq 1}\left( \frac{\sinh(vx/2)}{vx/2}\right)^{c_v} = \sum_{\lambda} \frac{\xi_{2\lambda} S_{2\lambda}}{|Aut(\lambda)|}\left(\frac{x}{2}\right)^{2|\lambda|} ,
    \end{split}
    \end{equation}
where $c_v=a_v$ if $v>1$ and $c_1=a_1-1$.

Note that in the case of $m=d$, we have
$$
H_{d,d}(k_1,\ldots,k_s;(d),\beta)=\overline{W}_{d,d}(k_1,\ldots,k_s;(d),\beta).
$$
Moreover, the case of $m=d$ reduces to one-part double Hurwitz numbers with completed cycles.
As a result, we obtain the result below which
generalizes the one-part cases studied in Nguyen~\cite{3-cycle}, and Chen and Wang~\cite{chw}. 

\begin{corollary}[One-part double Hurwitz numbers]\label{cor:double}
		Suppose, in the tuple $K=(k_1,\ldots, k_s)$, the largest number is $l$, and there are $b_i$ of $i$'s. Then, we have a formula for one-part double Hurwitz numbers with completed cycles:
	$$
	\begin{aligned}
		  H_{d,d}(K;(d),\beta)
		& = 
		\frac{d^{d'-1}}{Aut(\beta)}
		\prod_{i=1}^{l} \frac{b_i!}{(i!)^{b_i}}
		\sum_{t \geq 0}  [t_1^{b_1}\cdots t_l^{b_l} z^t y^t] \, \frac{t!}{d^t} \mathcal{U}(t_1,\ldots, t_l; z)\\  
		& \qquad \times  \qquad 
		 e^{-dy/2} y^{n-1}\sum_{\lambda} \frac{\xi_{2\lambda} S_{2\lambda}}{|Aut(\lambda)|}\left(\frac{y }{2}\right)^{2|\lambda|}  .
	\end{aligned}
	$$
\end{corollary}

\begin{corollary}[One-part double Hurwitz numbers]\label{cor:double-0}
		Suppose, in the tuple $K=(k_1,\ldots, k_s)$, the largest number is $l$, and there are $b_i$ of $i$'s. Then, we have the formula for one-part double Hurwitz numbers with completed cycles:
	$$
	\begin{aligned}
		&\quad  H_{d,d}(K;(d),\beta)\\
		& = 
		\frac{1}{ Aut(\beta) d}
		\prod_{i=1}^{l} \frac{b_i!}{(i!)^{b_i}}
		\sum_{h_{i,j} \geq 0,\, 1 \leq i \leq l\atop h_{i,0}+\cdots+h_{i,i-1}=b_i } \frac{ (\sum_{i,j}  jh_{i,j})!}{\prod_{i,j}  (h_{i,j}!)} \left\{ \prod_{i,j} \binom{i}{j}^{h_{i,j}}\right\}  d^{\sum_{i,j} (i-j)h_{i,j}} \\  
		& \qquad \times  \qquad 
		\left[y^{\sum_{i,j}  jh_{i,j}}\right] e^{-dy/2} y^{n-1}\sum_{\lambda} \frac{\xi_{2\lambda} S_{2\lambda}}{|Aut(\lambda)|}\left(\frac{y }{2}\right)^{2|\lambda|}  .
	\end{aligned}
	$$
\end{corollary}

\begin{proof}
	In Theorem~\ref{thm:explicit-formula}, setting $m=d$, we then obtain
	\begin{align*}
		& H_{d,d}(k_1,\ldots,k_s;(d),\beta)=\overline{W}_{d,d}(k_1,\ldots,k_s;(d),\beta)\\
	 = &
	\frac{|\mathcal{C}_{\beta}|}{dd!}
	\prod_{i=1}^{l} \frac{b_i!}{(i!)^{b_i}}\left[t_1^{b_1}\cdots t_{l}^{b_l}\right] 
	\sum_{h_{i,j} \geq 0 \atop 1 \leq i \leq l,\, 0 \leq j \leq i-1 } \frac{ \prod_{i}  t_{i}^{\sum_{j} h_{i,j}}}{\prod_{i,j}  (h_{i,j}!)} \left\{ \prod_{i,j} \binom{i}{j}^{h_{i,j}}\right\}  d^{\sum_{i,j} (i-j)h_{i,j}} \Omega, 
	\end{align*}
	where
	\begin{align*}
		\Omega 	& = 
		\left[\frac{y^{\sum_{i,j}  jh_{i,j}}}{(\sum_{i,j}  jh_{i,j})!}\right]  
		e^{-y/2}\frac{1}{e^{(d-1)y}} \prod_{v=1}^{d}\left\{ e^{y  v}-1 \right\}^{c_v}  \\
		& = 
		\left[\frac{y^{\sum_{i,j}  jh_{i,j}}}{(\sum_{i,j}  jh_{i,j})!}\right]  
		e^{-y/2}\frac{1}{e^{(d-1)y}} \prod_{v=1}^{d}\left\{ 2e^{vy/2}\sinh(vy/2)\right\}^{c_v}  \\
		& = 
		\left[\frac{y^{\sum_{i,j}  jh_{i,j}}}{(\sum_{i,j}  jh_{i,j})!}\right]  
		e^{-y/2}\frac{\prod^{n}_{f=1} \beta_f}{e^{(d-1)y}} y^{n-1}e^{(d-1)y/2}\prod_{v=1}^{d}\left\{\frac{\sinh(vy/2)}{vy/2}\right\}^{c_v} \\
		& = 
		(\prod^{n}_{f=1} \beta_f) 
		\left[\frac{y^{\sum_{i,j}  jh_{i,j}}}{(\sum_{i,j}  jh_{i,j})!}\right]  e^{-y/2}y^{n-1}e^{(d-1)(-y)/2}\prod_{v=1}^{d}\left\{\frac{\sinh(vy/2)}{vy/2}\right\}^{c_v}  \\
		& = 
		(\prod^{n}_{f=1} \beta_f) 
		\left[\frac{y^{\sum_{i,j}  jh_{i,j}}}{(\sum_{i,j}  jh_{i,j})!}\right]  
		e^{-dy/2}y^{n-1}\prod_{v=1}^{d}\left\{\frac{\sinh(vy/2)}{vy/2}\right\}^{c_v} \\
		&= (\prod^{n}_{f=1} \beta_f) 
		\left[\frac{y^{\sum_{i,j}  jh_{i,j}}}{(\sum_{i,j}  jh_{i,j})!}\right]  
		e^{-dy/2}y^{n-1}\sum_{\lambda} \frac{\xi_{2\lambda} S_{2\lambda}}{|Aut(\lambda)|}\left(\frac{y }{2}\right)^{2|\lambda|}  .
	\end{align*}
	Taking the coefficient of $\left[t_1^{b_1}\cdots t_{l}^{b_l}\right]$ completes the proof. 
\end{proof}

\begin{remark}
	In Nguyen~\cite{3-cycle}, an explicit formula similar to the one in Corollary~\ref{cor:double} for one-part double Hurwitz numbers with completed $3$-cycles, i.e., $k_1=\cdots = k_s=3$,
	was given. The author especially emphasized that one may not be able to obtain compact formulas for one-part double Hurwitz numbers with completed $r$-cycles
	for an arbitrary $r$. In Chen and Wang~\cite{chw}, the authors looked into this issue and solved the case for an arbitrary $r$.
\end{remark}

\section{Decomposition into hook-shape Hurwitz numbers}

\subsection{An algebra of bivariate polynomials over partitions}

Let $V_d$ be the vector space (over complex numbers) of polynomials in two indeterminates $x$ and $y$ which is
generated by the polynomials $x^d, \, x^{d-1}y, \, \ldots,\, x^i y^{d-i}, \, \ldots,\, y^d$.
Obviously, the dimension of $V_d$ is $d+1$.
Note that there is also a natural algebra structure over $\bigoplus_d V_d$.

For $\beta=(\beta_1, \ldots, \beta_n) \vdash d$, let
\begin{align}
	\rho(\beta) = \prod_{i=1}^{n}( x^{\beta_i}-y^{\beta_i}).
\end{align}
Clearly, $\rho(\beta) \in V_d$.
Let $\hat{V}_d$ be the vector space consisting of polynomials of the form
$$
\sum_{i=0}^d a_i x^i y^{d-i}
$$
where $\sum_{i=0}^d a_i=0$. It is easy to see $\hat{V}_d \subset V_d$,
and the dimension of $\hat{V}_d$ is strictly smaller than that of the space $V_d$.

\begin{lemma}
	For $\beta \vdash d$, $\rho(\beta) \in \hat{V}_d$.
\end{lemma}
\proof This follows from the fact that $\rho(\beta) \mid_{x=y=1}=0$. \qed

Let $\theta_{i,d} =[1^{d-i}, i] \vdash d$. We simply write $\theta_{i,d}$ as $\theta_i$ when
$d$ is clear from the context.

\begin{lemma}\label{lem: base1}
	For $d \geq 1$,	$\{ \rho(\theta_{1,d}),\rho(\theta_{2,d}), \ldots,  \rho(\theta_{d,d}) \}$ is a basis of $\hat{V}_d$.
\end{lemma}

\begin{proof}
	We prove by induction.
	For $d = 1$, it is obvious.
	Suppose it holds for $d=N\geq 1$.
	For $d=N+1$, we first notice that for $0 < i < N+1$,
	$$
	\frac{\rho(\theta_{i,N+1})}{x-y}= \rho(\theta_{i,N}).
	$$
	Since $[1^{N-i}, \, i] \vdash N$ and $\rho(\theta_{i,N})$ ($0<i <N+1$) give a basis of $\hat{V}_{N}$
	by assumption, $\rho(\theta_{i,N+1})$ ($0<i <N+1$) are linearly independent in $\hat{V}_{N+1}$.
	Secondly, $\rho(\theta_{N+1,N+1})$ is not divisible by $(x-y)^2$ while $\rho(\theta_{i,N+1})$ is divisible by $(x-y)^2$ for any $0<i <N+1$.
	Thus, $\rho(\theta_{N+1,N+1})$ can not be written as a linear combination of $\rho(\theta_{i,N+1})$ for $0< i < N+1$.
	Hence, $\rho(\theta_{i,N+1})$ ($0 < i \leq N+1$) are linearly independent.
	Since the dimension of $\hat{V}_{N+1}$ is at most $N+1$,
	$F(\theta_{i,N-1})$ ($0 < i \leq N+1$) give a basis for $\hat{V}_{N+1}$.
	This completes the proof.
\end{proof}

Due to Lemma 5.3, any $\rho(\beta)$ is a linear combination of $\rho(\theta_{i,d})$.
Next, we will discuss the structure of this representation.
Let
\begin{align}
	\rho^{k}_{\beta} : = \frac{\mathrm d^{k}}{\mathrm d x^{k}} \prod_{i= 1}^{n} \left(x^{\beta_{i}} - 1 \right) \bigg |_{x = 1}.
\end{align}

\begin{lemma}\label{lem:coeff-beta}
For any $\beta=(\beta_1, \ldots, \beta_n)$, there exists
	\begin{align}
		\rho^k_{\beta}=&
		\sum_{ b_1+\cdots+b_n = k, \atop b_i\geq 1} \binom{k}{b_1,\ldots,b_n}(\beta_1)_{b_1} (\beta_2)_{b_2}\cdots(\beta_n)_{b_n}.
	\end{align}
In particular,
\begin{align}
		\rho^k_{\theta_i}= \begin{cases}
			(k)_{d-i} (i)_{k-d+i}, & \mbox{if $k> d-i$}, \\
			0, &  \mbox{otherwise.}
		\end{cases}
	\end{align}
\end{lemma}

\begin{proof}
	For a general $\beta = (\beta_1,\ldots,\beta_n)$, we first get:
	$$
	\begin{aligned}
		\rho^{k}_{(\beta_1,\ldots,\beta_n)} 
		=&
		\sum_{ b_1+\cdots+b_n = k, \atop b_i\geq 1}\binom{k}{b_1,\ldots,b_n}(x^{\beta_1} - 1)^{(b_1)}(x^{\beta_2} - 1)^{(b_2)}\cdots(x^{\beta_n} - 1)^{(b_n)} \bigg |_{x = 1} \\
		=&
		\sum_{ b_1+\cdots+b_n = k, \atop b_i\geq 1} \binom{k}{b_1,\ldots,b_n}(\beta_1)_{b_1} (\beta_2)_{b_2}\cdots(\beta_n)_{b_n},
	\end{aligned}
	$$
where $f(x)^{(k)}=\frac{\mathrm d^k}{\mathrm d x^k}f(x)$. For $\beta = \theta_i$, 
the last number obviously equals $(k)_{d-i} (i)_{k-d+i}$ if $k>d-i$, and $0$ otherwise,
completing the proof.
\end{proof}

\begin{theorem}\label{thm:rho-decomp}
	
	For any $\beta \vdash d$, we have $(a_{d,\beta},a_{d-1, \beta},\ldots,a_{1,\beta})$ such that	
	$$
	\rho(\beta) = a_{d,\beta}\rho(\theta_{d})+a_{d-1,\beta}\rho(\theta_{d-1})+\cdots+ a_{1,\beta}\rho(\theta_{1}),
	$$
	and the coefficients satisfy:
	
	$$
	\begin{bmatrix}
		a_{d,\beta} \\
		a_{d-1,\beta}\\
		\vdots \\
		a_{1,\beta} 
	\end{bmatrix}
	=
	\begin{bmatrix}
		\rho^{1}_{\theta_{d}} & \ & \ & \ \\
		\rho^{2}_{\theta_{d}} & \rho^{2}_{\theta_{d-1}} & \ & \ \\
		\vdots &  \vdots & \ddots & \ \\
		\rho^{d}_{\theta_{d}} & \rho^{d}_{\theta_{d-1}} & \cdots &  \rho^{d}_{\theta_{1}} 
	\end{bmatrix}^{-1}
	\times
	\begin{bmatrix}
		\rho^{1}_{\beta} \\
		\rho^{2}_{\beta} \\
		\vdots \\
		\rho^{d}_{\beta}  
	\end{bmatrix}.
	$$
\end{theorem}

\begin{proof}
	
Suppose
	$$
	\rho(\beta) = a_{d,\beta}\rho(\theta_{d})+a_{d-1,\beta}\rho(\theta_{d-1})+\cdots+ a_{1,\beta}\rho(\theta_{1}).
	$$
	Let $y=1$ and differentiate both sides of the above equation.
	Setting $x=1$ next, the right side of the resulted equation becomes $a_{d,\beta}\rho^{1}_{\theta_{d}}$. So we can calculate $a_{d,\beta}$. 
	Let $y=1$ and differentiate both sides of the above equation twice.
	Setting $x=1$, the right side of the resulted equation becomes $a_{d,\beta}\rho^{2}_{\theta_{d}} + a_{d-1,\beta}\rho^{2}_{\theta_{d-1}}$.
	From this, we can subsequently obtain $a_{d-1,\beta}$. Continuing this process,
we obtain the equations:
	\begin{itemize}
		\item $\rho^{1}_{\beta} = a_{d,\beta}\rho^{1}_{\theta_{d}}$
		\item $\rho^{2}_{\beta} = a_{d,\beta}\rho^{2}_{\theta_{d}}+ a_{d-1,\beta}\rho^{2}_{\theta_{d-1}}$
		\item $\cdots $
		\item $\rho^{d}_{\beta}= a_{d,\beta}\rho^{d}_{\theta_{d}}+ a_{d-1 ,\beta}\rho^{d}_{\theta_{d-1}}+\cdots+a_{1,\beta} \rho^{d}_{\theta_{1}}$.
	\end{itemize}
	Equivalently, we have the matrix expression below:
	$$
	\begin{bmatrix}
		\rho^{1}_{\theta_{d}} & \ & \ & \ \\
		\rho^{2}_{\theta_{d}} & \rho^{2}_{\theta_{d-1}} & \ & \ \\
		\vdots &  \vdots & \ddots & \ \\
		\rho^{d}_{\theta_{d}} & \rho^{d}_{\theta_{d-1}} & \cdots &  \rho^{d}_{\theta_{1}} 
	\end{bmatrix}
	\times
	\begin{bmatrix}
		a_{d,\beta} \\
		a_{d-1,\beta}\\
		\vdots \\
		a_{1,\beta} 
	\end{bmatrix}
	=
	\begin{bmatrix}
		\rho^{1}_{\beta} \\
		\rho^{2}_{\beta} \\
		\vdots \\
		\rho^{d}_{\beta}  
	\end{bmatrix}.
	$$
	Note that none of the diagonal entries $\rho_{\theta_i}^{d+1-i}$ are zero.
	Thus, the left triangular matrix is invertible, and the proof follows. 
\end{proof}

\begin{proposition}
	For $\beta \vdash d$, we have
	\begin{itemize}
		\item For $d\geq i \geq d-\ell(\beta)+2$, $a_{i,\beta} =0$;
		\item For $d-i-\ell(\beta)=0 \mod 2$, $a_{i,\beta}=0$.
	\end{itemize}
\end{proposition}
\begin{proof}
	According to Lemma~\ref{lem:coeff-beta}, $\rho_{\beta}^k =0$ if $k< \ell(\beta)$.
	As a result, the first item follows.

	For the second item, we first observe that
	\begin{itemize}
		\item The $\rho(\beta)$ is symmetric in $x$ and $y$ if and only if $\ell(\beta)$ is even;
		\item The $\rho(\beta)$ is antisymmetric in $x$ and $y$ if and only if $\ell(\beta)$ is odd.
	\end{itemize}
Accordingly,
 $\rho(\theta_{1,d}),\rho(\theta_{3,d}),\ldots$ are antisymmetric, and $\rho(\theta_{2,d}),\rho(\theta_{4,d}),\ldots$ are symmetric, if $d$ is odd.
 The case for $d$ being even is analogous.
 It is also not difficult to see that a symmetric (resp.~antisymmetric) polynomial can be only written as a linear combination of symmetric (resp.~antisymmetric) polynomials.
As such, the second item follows. 
\end{proof}

Taking Proposition 5.1 into consideration and with an analogous argument as Theorem 5.1, we have the following corollary.

\begin{corollary}
	
	For every $\beta \vdash d$ with $\ell(\beta) = c$, we have 
	$$
	\rho(\beta) = a_{d-c+1, \beta}\rho(\theta_{d-c+1})+a_{d-c-1, \beta}\rho(\theta_{d-c-1})+\cdots+ a_{1+\delta_{(-1)^{d-c},-1}, \beta}\rho(\theta_{1+\delta_{(-1)^{d-c},-1}})
	$$
	and the coefficients satisfy
	$$
	\begin{bmatrix}
		a_{d-c+1,\beta} \\
		a_{d-c-1,\beta}\\
		\vdots \\
		a_{1 + \delta_{(-1)^{d-c},-1},\beta} 
	\end{bmatrix}
	=\begin{bmatrix}
		\rho^{c}_{\theta_{d-c+1}} & \ & \ & \ \\
		\rho^{c+2}_{\theta_{d-c+1}} & \rho^{c+2}_{\theta_{d-c-1}} & \ & \ \\
		\vdots &  \vdots & \ddots & \ \\
		\rho^{d-\delta_{(-1)^{d-c},-1}}_{\theta_{d-c+1}} & \rho^{d-\delta_{(-1)^{d-c},-1}}_{\theta_{d-c-1}} & \cdots &  \rho^{d-\delta_{(-1)^{d-c},-1}}_{\theta_{1+\delta_{(-1)^{d-c},-1}}} 
	\end{bmatrix}^{-1}
	\times
	\begin{bmatrix}
		\rho^{c}_{\beta} \\
		\rho^{c+2}_{\beta} \\
		\vdots \\
		\rho^{d-\delta_{(-1)^{d-c},-1}}_{\beta}  
	\end{bmatrix}. 
	$$
\end{corollary}

We remark that there may be other interesting bases for $\hat{\Lambda}_d$.
For example, 
{$$\{(d),(1, d-1), \ldots, (\lfloor d/2 \rfloor, d-\lfloor d/2 \rfloor), (1,2,d-3), (1,3, d-4),\ldots\}
$$}
is also a basis.
These will be left for future investigation.

\subsection{Decomposition theorems}

As an application of the theory developed above, we can decompose
an arbitrary one-part quasi-triple Hurwitz number into simpler ones, termed hook-shape Hurwitz numbers.

\begin{theorem}[Recurrence for $\beta$]\label{thm:decomp-beta}
	For any $\beta \vdash d$, we have
	\begin{align}
		H_{d,m}(K;\beta)= a_{d,\beta} H_{d,m}(K;\theta_{d,d})+a_{d-1,\beta} H_{d,m}(K;\theta_{d-1,d})+\cdots+ a_{1,\beta} H_{d,m}(K;\theta_{1,d}).
	\end{align}
\end{theorem}
\begin{proof}
	From Theorem~\ref{thm:explicit-formula}, we first have
	$$
	\begin{aligned}
		& \quad \overline{W}_{d,m}(K;(d),\beta)\\
		& = 
		\frac{(m-1)!|\mathcal{C}_{\beta}|\binom{d-1}{m-1}}{m!d!}
		\prod_{i=1}^{l} \frac{b_i!}{(i!)^{b_i}}\sum_{h_{i,j} \geq 0,\, 1 \leq i \leq l\atop h_{i,0}+\cdots+h_{i,i-1}=b_i } \frac{ (\sum_{i,j}  jh_{i,j})!}{\prod_{i,j}  (h_{i,j}!)} \left\{ \prod_{i,j} \binom{i}{j}^{h_{i,j}}\right\}  d^{\sum_{i,j} (i-j)h_{i,j}} \\  
		& \qquad \times \qquad
		\frac{\mathrm d^{d-m}}{\mathrm d  x^{d-m}} 
		\left[y^{\sum_{i,j}  jh_{i,j}}\right]  
		e^{-y/2}
		(x-e^{-y})^{-1}\prod_{v=1}^{n}\left\{ x^{\beta_v}-e^{-y\beta_v} \right\}
		\bigg |_{x=1} .\\
	\end{aligned}
	$$
	Applying Theorem~\ref{thm:rho-decomp} to the factor $\prod_{v=1}^{n} ( x^{\beta_v}-e^{-y\beta_v})$
	and taking advantage of the linearity of the operators of taking derivative, taking a coefficient and
	evaluating at a point next leads to
		\begin{align*}
		\overline{W}_{d,m}(K;\beta)= a_{d,\beta} \overline{W}_{d,m}(K;\theta_{d,d})+a_{d-1,\beta} \overline{W}_{d,m}(K;\theta_{d-1,d})+\cdots+ a_{1,\beta} \overline{W}_{d,m}(K;\theta_{1,d}).
	\end{align*}
Noticing that
$$
	H_{d,m}(K;\beta) =\sum_{k=0}^{d-m} (-1)^k \stirlingI{m+k}{m} \overline{W}_{d,m+k}(K;\beta),
$$
and the coefficients $a_{d, \beta}$ do not depend on $k$ (or $m+k$) will complete the proof. 
	\end{proof}

Similarly, we can obtain a recurrence with respect to $K$.
Let
$$
\begin{aligned}
	\widetilde{H}_{d,m}(K;\beta)=H_{d,m}(K;\beta)\prod_{i=1}^{s} k_{i}!, \qquad \widetilde W_{d,m}(K;\beta) = \overline{W}_{d,m}(K;\beta)  \prod_{i=1}^{s} k_{i}!.
\end{aligned}
$$

\begin{theorem}[Recurrence for $K$] \label{thm:decomp-K}
For any $(k_1,\ldots,k_{s}) \vdash d'$, we have
	\begin{align}
	\widetilde  W_{d,m}(K;\beta) &
	= 	a_{d',K}\widetilde{W}(\theta_{d',d'};\beta)+a_{d'-1,K}\widetilde{W}(\theta_{d'-1,d'};\beta)+\cdots+ a_{1,K}\widetilde{W}(\theta_{1,d'};\beta), \\
		\widetilde{H}_{d,m}(K;\beta) & 
	= 
	a_{d',K}\widetilde{H}_{d,m}(\theta_{d',d'};\beta)+a_{d'-1,K}\widetilde{H}_{d,m} (\theta_{d'-1,d'};\beta)+\cdots +a_{1,K}\widetilde{H}_{d,m} (\theta_{1,d'};\beta).
\end{align}
\end{theorem}
\begin{proof}
According to Theorem~\ref{thm:main-recur}, we know that
	$$
	\begin{aligned}
		\overline{W}_{d,m}(K;\beta)= &
		\frac{1}{m!d!}\sum_{\lambda \vdash d}\left(\prod_{i=1}^{s} \frac{p_{k_i}(\lambda)}{k_i!}\right) \frac{|\mathcal{C}_{(d)}|\chi_{(d)}^
			{\lambda}|\mathcal{C}_{\beta}|\chi_{\beta}^{\lambda}}{\dim(\lambda)}\mathfrak{c}_{\lambda,m}. \\
	\end{aligned}
	$$
Then, we have:
	$$
	\begin{aligned}
		\widetilde  W(K;(d),\beta) 
		= &
		\frac{1}{m!d!}\sum_{\lambda \vdash d}\left\{\prod_{i=1}^{s} p_{k_i}(\lambda)\right\} \frac{|\mathcal{C}_{(d)}|\chi_{(d)}^
			{\lambda}|\mathcal{C}_{\beta}|\chi_{\beta}^{\lambda}}{\dim(\lambda)}\mathfrak{c}_{\lambda,m} \\
		= &
		\frac{1}{m!d!}\sum_{k = 0}^{d-1}|\mathcal{C}_{(d)}||\mathcal{C}_{\beta}|\left\{\prod_{i=1}^{s} p_{k_i}[1^k,d-k]\right\} \frac{(-1)^k\chi_{\beta}^{[1^k,d-k]}}{\dim([1^k,d-k])}\mathfrak{c}_{[1^k,d-k],m}. \\
	\end{aligned}
	$$
By definition of $p_i(\lambda)$, we next have
	$$
	p_i([1^k, d-k])=(d-k-1/2)^i-(-k-1/2)^i.
	$$
Let 
	$
\widetilde\rho_k(K)=	\prod_{i=1}^{s} p_{k_i}([1^k,d-k])
	$  for $ 0 \leq k\leq d-1$.
	It is obvious that
	$$
	\widetilde\rho_k(K) = \rho(K) \Big|_{x= d-k-1/2, \, y = -k-1/2}.
	$$
As a result, we have
	$$
\widetilde\rho_k(K)= a_{d',K}\widetilde\rho_k(\theta_{d',d'})+a_{d'-1,K}\widetilde\rho_k(\theta_{d'-1,d'})+\cdots+ a_{1,K}\widetilde\rho_k(\theta_{1,d'}).
$$	
Again, since $a_{i,K}$ do not depend on other variables, we easily arrive
at the equations in the theorem by exploring linearity. 		
\end{proof}

In view of Theorem~\ref{thm:decomp-beta} and Theorem~\ref{thm:decomp-K}, the calculation of all triple Hurwitz numbers comes down to 
calculating $\widetilde{W}([1^{d'-i},i];(d),[1^{d-j},j]) $ of one-part triple Hurwitz numbers which have simpler and elementary expression shown below.

\begin{theorem}[Hook-shape Hurwitz numbers]
There exists
	\begin{align}
		\begin{aligned}
			\widetilde{W}_{d,m}([1^{d'-i},i];(d),[1^{d-j},j])   
			&= 
			\frac{|\mathcal{C}_{[1^{d-j},j]}|d^{d'-i-1}}{m!}
			\sum_{k=0}^{d-1} (-1)^k \left[ (d-k-1/2)^i-(-k-1/2)^i\right] \\
			& \qquad \times \bigg[\binom{d-j-1}{k} + \binom{d-j-1}{k-j} \bigg]{d-k-1 \choose d-m} .
		\end{aligned}
	\end{align}
\end{theorem}
\begin{proof}
We first have
	$$
	\begin{aligned}
		& \quad \widetilde{W}([1^{d'-i},i];(d),[1^{d-j},j]) \\
		& 
		=\frac{1}{m!d!}\sum_{\lambda \vdash d}|\mathcal{C}_{(d)}||\mathcal{C}_{[j,1^{d-j}]}| p_{i} (\lambda)d^{d'-i} \frac{\chi_{(d)}^
			{\lambda}\chi_{[1^{d-j},j]}^{\lambda}}{\dim(\lambda)}\mathfrak{c}_{\lambda , m} \\
		& = 
		\frac{(d-1)!|\mathcal{C}_{[j,1^{d-j}]}|}{m!d!}
		\sum_{k=0}^{d-1}p_{i}[1^k,d-k]d^{d'-i}
		(-1)^k \chi_{[1^{d-j},j]}^{[1^k,d-k]}\frac{\binom{d-1}{m-1}\binom{m-1}{k}}{\dim([1^k,d-k
			])} \\
		& = 
		\frac{|\mathcal{C}_{[j,1^{d-j}]}|\binom{d-1}{m-1}}{m!d!}
		\sum_{k=0}^{d-1}p_{i}[1^k,d-k]d^{d'-i}
		(-1)^k\chi_{[1^{d-j},j]}^{[1^k,d-k]}k!(d-1-k)!\binom{m-1}{k} \\ 
		& = 
		\frac{|\mathcal{C}_{[j,1^{d-j}]}|\binom{d-1}{m-1}}{md!}
		\sum_{k=0}^{d-1}p_{i}[1^k,d-k]d^{d'-i}
		(-1)^k\chi_{[1^{d-j},j]}^{[1^k,d-k]}(d-k-1)_{d-m} \\ 
		& = 
		\frac{|\mathcal{C}_{[j,1^{d-j}]}|\binom{d-1}{m-1}d^{d'-i}}{md!}
		\sum_{k=0}^{d-1}\left( (d-k-1/2)^i-(-k-1/2)^i\right)
		(-1)^k\chi_{[1^{d-j},j]}^{[1^k,d-k]}(d-k-1)_{d-m} \\ 
	\end{aligned}
	$$

According to the Murnaghan-Nakayama Rule, we next compute
	$$
	\begin{aligned}
		\chi_{[1^{d-j},j]}^{[1^k,d-k]} = & \chi_{[1^{d-j}]}^{[1^k,d-k-j]} + \chi_{[1^{d-j}]}^{[1^{k-j},d-k]} \\
		= & \dim([1^k,d-k-j]) + \dim([1^{k-j},d-k]) \\
		= & \binom{d-j-1}{k} + \binom{d-j-1}{k-j} .
	\end{aligned}
	$$
Plugging it into the last formula, we obtain	
$$
\begin{aligned}
	&\quad  \widetilde{W}([1^{d'-i},i];(d),[1^{d-j},j]) \\
	& = 
\frac{|\mathcal{C}_{[j,1^{d-j}]}|d^{d'-i-1}}{m!}
\sum_{k=0}^{d-1} (-1)^k \left[ (d-k-1/2)^i-(-k-1/2)^i\right]\\
 & \qquad \times   \bigg[\binom{d-j-1}{k} + \binom{d-j-1}{k-j} \bigg]{d-k-1 \choose d-m},
\end{aligned}
$$	
 completing the proof. 
\end{proof}

\section{Polynomiality}

As mentioned in the section Introduction, certain scaled single Hurwitz numbers are polynomial
in the parts of the profile of the single nonsimple branch point.
This may be the starting point of looking into the polynomiality-like propertites of Hurwitz numbers.
Subsequently, piecewise polynomiality for a variety of Hurwitz numbers has been closely investigated since the piecewise
polynomiality of (standard) double Hurwitz numbers proved by Goulden, Jackson and Vakil~\cite{GJV05}.
Regarding completed cycles, Shadrin, Spitz and Zvonkine~\cite{complete11} obtained the piecewise polynomiality for double Hurwitz numbers
with completed $r$-cycles for a fixed $r$.

For our studied triple Hurwitz numbers, we also obtain a piecewise polynomiality result in the conventional sense below.

\begin{corollary}[Piecewise polynomiality]\label{thm:piece-poly}
For fixed $k_1, \ldots, k_s$, $m$, $n$, and $d$, there exists
a polynomial $Pol$ in variables $\beta_1, \ldots, \beta_n$ of degree at most $2g$ (determined by eq.~\eqref{eq:dimen-constraint}) such that
$$
H_{d,m} (k_1,\ldots,k_s;(d),\beta) \times d \times \prod_i a_i ! = Pol(\beta_1,\ldots, \beta_n)
$$
for any $\beta=(\beta_1,\ldots, \beta_n)= [1^{a_1}, \ldots, d^{a_d}]  \vdash d$.
\end{corollary}

\begin{remark}
	Since for triple Hurwitz numbers, it is piecewise polynomiality instead of real polynomiality, we will see a different version shortly.
\end{remark}

\subsection{Double piecewise polynomiality}

 We now present another result regarding piecewise polynomiality,
where the piecewise polynomiality is not only in the parts of the profile of a branch point as usual,
but also in the ``orders" of the completed cycles.

\begin{theorem}
For any fixed $d,d',m,n>0$, there exists a polynomial $Pc(k_1,\ldots, k_s, \beta_1,\ldots, \beta_n)$ in $k_1,\ldots, k_s$, $\beta_1,\ldots, \beta_n$ {with degree between $0$ and $d+d'$}
such that for any $k_1+\cdots+k_s= d'$ and $\beta_1+\cdots + \beta_n =d$, we have
$$
Pc(k_1,\ldots, k_s, \beta_1,\ldots, \beta_n)= \widetilde{H}_{d,m}(k_1,\ldots, k_s; \beta_1,\ldots, \beta_n).
$$

\end{theorem}

\proof From the decomposition theorems discussed in the last section,
we know that there exists the following form
$$
\widetilde{H}_{d,m}(k_1,\ldots, k_s; \beta_1,\ldots, \beta_n)= \sum_i \sum_j C_{i,j} \widetilde{H}_{d,m}(\theta_{d',i}; \theta_{d,j}),
$$
where $C_{ij}= u(i)a_{i,K} a_{j,\beta}$.
Noticing that $a_{i,K}$ is a polynomial in $k_1,\ldots, k_s$ with degree no greater then $d'$ and $a_{j,\beta}$
is a polynomial in $\beta_1,\ldots, \beta_n$ with degree no greater then $d$, the proof follows. \qed

Note that the highest polynomial degree is different.
When $m=d$, we obtain real polynomiality which is discussed in the remaining subsections.

\subsection{An analogue of the $\lambda_g$ conjecture}

As a consequence of the ELSV formula~\cite{ELSV}, the intersection numbers, often expressed in the Witten symbols, $\langle \cdots \rangle_g$, are connected to the coefficients of the single
Hurwitz numbers:
\begin{align}
	\langle \tau_{b_1}\cdots \tau_{b_n} \lambda_k \rangle_g:= \int_{\overline{{M}}_{g,n}} \psi_1^{b_1} \ldots \psi_n^{b_n} \lambda_k=(-1)^k
	[\beta_1^{b_1}\cdots \beta_n^{b_n}] \mathsf{H}^g_{\beta} \bigg/t! \prod_{i=1}^n \frac{\beta_i^{\beta_i}}{\beta_i !}.
\end{align}
Specially, the $\lambda_g$-conjecture~\cite{FP} is concerned with an explicit formula for $\langle \tau_{b_1}\cdots \tau_{b_n} \lambda_g \rangle_g$
where $b_1+\cdots+b_n=2g-3+n$, i.e., the lowest degree term:
\begin{align}
	\langle \tau_{b_1}\cdots \tau_{b_n} \lambda_g \rangle_g =\int_{\overline{{M}}_{g,n}} \psi_1^{b_1} \ldots \psi_n^{b_n} \lambda_g={2g-3+n \choose b_1,\ldots, b_n} \frac{2^{2g-1}-1}{2^{2g-1}(2g)!} |B_{2g}|.
\end{align}
Thus, the connection between the Hurwitz numbers and Hodge integrals over space of curves can help understand
each other erea better.
Since then, a number of ELSV-type formulas for various of Hurwitz numbers have been conjectured or proved, e.g.~\cite{DL22,GJV05,r-elsv}.
As for Hurwitz numbers with completed cycles, Shadrin, Spitz and Zvonkine~\cite{complete11} proved the polynomiality of one-part
double Hurwitz numbers with completed $r$-cycles (i.e., $k_1=\cdots k_s=r$) and conjectured a ELSV-type formula for them.

In the following, we first show the polynomiality of one-part double Hurwitz number
with any combination of completed cycles, and then prove an analogue of the $\lambda_g$
conjecture.

\begin{theorem}[Polynomiality] \label{thm:poly-double}
	For fixed $n>0$ and $k_1,\ldots, k_s$ with $s\geq 1$, the following quantity induced by one-part double Hurwitz number with completed cycles
		$$
	H_{|\beta|,|\beta|} (k_1,\ldots, k_s; \beta_1,\ldots, \beta_n) \times Aut(\beta) \times |\beta|
	$$
	is a polynomial in $\beta_1, \ldots, \beta_n$, with the lowest degree $s $ and the highest degree $(\sum_{i=1}^{s}k_i)-n+1$.
\end{theorem}
\begin{proof}
According to Corollary~\ref{cor:double}, we first have
	$$
	\begin{aligned}
		& \quad dH_{d,d} (k_1,\ldots, k_s; \beta_1,\ldots, \beta_n) \times Aut(\beta)= d\overline{W}_{d,d}(k_1,\ldots,k_s;(d),\beta)\times Aut(\beta)\\
		& = 
		\frac{ Aut(\beta)|\mathcal{C}_{\beta}|\prod_{f=1}^{n} \beta_f}{d!}
		\prod_{i=1}^{l} \frac{b_i!}{(i!)^{b_i}}
		\sum_{h_{i,j} \geq 0,\, 1 \leq i \leq l\atop h_{i,0}+\cdots+h_{i,i-1}=b_i } \frac{ (\sum_{i,j}  jh_{i,j})!}{\prod_{i,j}  (h_{i,j}!)} \left\{ \prod_{i,j} \binom{i}{j}^{h_{i,j}}\right\}  d^{\sum_{i,j} (i-j)h_{i,j}} \\  
		& \qquad \times  \qquad 
		\left[y^{\sum_{i,j}  jh_{i,j}}\right]e^{-dy/2}  y^{n-1}\sum_{\lambda} \frac{\xi_{2\lambda} S_{2\lambda}}{|Aut(\lambda)|}\left(\frac{y }{2}\right)^{2|\lambda|} \\
		& = 
		\prod_{i=1}^{l} \frac{b_i!}{(i!)^{b_i}}
		\sum_{h_{i,j} \geq 0,\, 1 \leq i \leq l\atop h_{i,0}+\cdots+h_{i,i-1}=b_i } \frac{ (\sum_{i,j}  jh_{i,j})!}{\prod_{i,j}  (h_{i,j}!)} \left\{ \prod_{i,j} \binom{i}{j}^{h_{i,j}}\right\}  d^{\sum_{i,j} (i-j)h_{i,j}} \\  
		& \qquad \times  \qquad 
		\left[y^{\sum_{i,j}  jh_{i,j}}\right] e^{-dy/2} y^{n-1}\sum_{\lambda} \frac{\xi_{2\lambda} S_{2\lambda}}{|Aut(\lambda)|}\left(\frac{y }{2}\right)^{2|\lambda|}. \\
	\end{aligned}
	$$
Writing $d=\beta_1+\cdots +\beta_n$ and recalling $S_{2\lambda}=\prod_i S_{2 \lambda_i}$ with $S_{2j}=(-1+\beta_1^{2j}+\cdots+ \beta_n^{2j})$, we
conclude that 	
$$
H_{|\beta|,|\beta|} (k_1,\ldots, k_s; \beta_1,\ldots, \beta_n) \times Aut(\beta) \times |\beta|
$$
is indeed a polynomial of $\beta_i$.

Next, we consider the lowest degree of the polynomial.
First, since $i-j>0$, we have 
$$
\sum_{i,j} (i-j)h_{i,j} \geq \sum_{i,j} h_{i,j}=\sum_i {b_i}= s,
$$
where the equality is achieved if and only if $h_{i,0}=\cdots=h_{i,i-2}=0$ and $h_{i,i-1}=b_i$ for every $i$.
Note that
the lowest degree terms in $\beta_i$ contributed from the following factor
$$
\left[y^{\sum_{i,j}  jh_{i,j}}\right] e^{-dy/2} y^{n-1}\sum_{\lambda} \frac{\xi_{2\lambda} S_{2\lambda}}{|Aut(\lambda)|}\left(\frac{y }{2}\right)^{2|\lambda|}
$$
are always degree zero terms since the existence of the shift ``$-1$" in $S_{2j}$.
Thus, the lowest degree is $s$ contributed from the factor $d^{\sum_{i,j} (i-j)h_{i,j}}$
if the coefficient of term $d^s$ is not zero.
The last condition will be confirmed in an analogue of the $\lambda_g$ conjecture
proved later.

Finally, we determine the highest degree.
It is not difficult to observe that
the highest possible degree in $\beta_i$ contributed from the following factor
$$
\left[y^{\sum_{i,j}  jh_{i,j}}\right] e^{-dy/2} y^{n-1}\sum_{\lambda} \frac{\xi_{2\lambda} S_{2\lambda}}{|Aut(\lambda)|}\left(\frac{y }{2}\right)^{2|\lambda|}
$$
is $\sum_{i,j}  jh_{i,j} -n+1$.
As a result, the highest possible degree of the studied quantity is 
$$
\sum_{i,j} (i-j)h_{i,j} + \sum_{i,j}  jh_{i,j} -n+1 = \sum_i k_i -n+1.
$$

In order to conclude that the highest degree is exactly $\sum_i k_i -n+1$,
it suffices to show the coefficient of the term $ \beta_1^{\sum_i k_i -n+1}$ is not zero.
The desired coefficient is easily seen to be
$$
\begin{aligned}
	&  
	\prod_{i=1}^{l} \frac{b_i!}{(i!)^{b_i}}
	\sum_{h_{i,j} \geq 0,\, 1 \leq i \leq l\atop h_{i,0}+\cdots+h_{i,i-1}=b_i } \frac{ (\sum_{i,j}  jh_{i,j})!}{\prod_{i,j}  (h_{i,j}!)} \left\{ \prod_{i,j} \binom{i}{j}^{h_{i,j}}\right\} \\  
	& \qquad \times  \qquad 
	\left[y^{\sum_{i,j}  jh_{i,j}}\right] e^{-y/2} y^{n-1}\sum_{\lambda} \frac{\xi_{2\lambda}}{|Aut(\lambda)|}\left(\frac{y }{2}\right)^{2|\lambda|}\\
\end{aligned}
$$
Based on structure:
$$
\begin{aligned}
	\prod_{v\geq 1}\left( \frac{\sinh(vx/2)}{vx/2}\right)^{c_v} =\sum_{\lambda} \frac{\xi_{2\lambda} S_{2\lambda}}{|Aut(\lambda)|}\left(\frac{x}{2}\right)^{2|\lambda|},
\end{aligned}
$$
we obtain for $c_1=1$ and $c_v=0$ if $v>1$ that 
$$
\begin{aligned}
	\frac{\sinh(x/2)}{x/2}& = \sum_{\lambda} \frac{\xi_{2\lambda} }{|Aut(\lambda)|}\left(\frac{x}{2}\right)^{2|\lambda|}\\
& =\sum_{i=0}^{\infty} \frac{x^{2i}}{2^{2i}(2i+1)!}.
\end{aligned}
$$
Thus for $\sum_{i,j}  jh_{i,j} \geq n-1$, we have 
$$
\begin{aligned}
& \quad \left[y^{\sum_{i,j}  jh_{i,j}}\right] e^{-y/2} y^{n-1}\sum_{\lambda} \frac{\xi_{2\lambda}}{|Aut(\lambda)|}\left(\frac{y }{2}\right)^{2|\lambda|} \\
&=
\begin{cases} 
\frac{1}{2^{2k}}\left(\frac{1}{0!(1+2k)!}+\frac{1}{2!(2k-1)!}+\cdots+\frac{1}{(2k)!1!}\right), & \mbox{if $2k=\sum_{i,j}  jh_{i,j}-n+1$},\\
\frac{-1}{2^{2k+1}}\left(\frac{1}{1!(2k+1)!}+\frac{1}{3!(2k-1)!}+\cdots+\frac{1}{(2k+1)!1!}\right), & \mbox{if $2k+1=\sum_{i,j}  jh_{i,j}-n+1$},
\end{cases} \\
& = \frac{(-1)^{\sum_{i,j}  jh_{i,j}-n+1}}{(1\sum_{i,j}  jh_{i,j}-n+2)!}.
\end{aligned}
$$
As a consequence, the coefficient of the term $ \beta_1^{\sum_i k_i -n+1}$ equals
  	$$
  	\begin{aligned}
  & 
 \prod_{i=1}^{l} \frac{b_i!}{(i!)^{b_i}}
		\sum_{h_{i,j} \geq 0,\, 1 \leq i \leq l\atop h_{i,0}+\cdots+h_{i,i-1}=b_i } \left\{\prod_{i,j} \frac{ \binom{i}{j}^{h_{i,j}}}{  (h_{i,j}!)} \right\} (\sum_{i,j}  jh_{i,j})!  
 \frac{(-1)^{\sum_{i,j}  jh_{i,j}-n+1}}{(\sum_{i,j}  jh_{i,j}-n+2)!}\\
   = &
 \prod_{i=1}^{l} \frac{b_i!}{(i!)^{b_i}}
		\sum_{h_{i,j} \geq 0,\, 1 \leq i \leq l\atop h_{i,0}+\cdots+h_{i,i-1}=b_i } \left\{\prod_{i,j} \frac{ \binom{i}{j}^{h_{i,j}}}{  (h_{i,j}!)} \right\} (\sum_{i,j}  jh_{i,j})_{n-1}  
 \frac{(-1)^{\sum_{i,j}  jh_{i,j}-n+1}}{\sum_{i,j}  jh_{i,j}-n+2}.\\
	\end{aligned}
	$$
Subsequently, it is critical to realize that the last sum equals
$$
\begin{aligned}
&(-1)^{n-1}\int_{-1/2}^{1/2}\left(\frac{\mathrm d^{n-1}}{\mathrm d k^{n-1}}\left\{\frac{\prod_{i=1}^{s}p_{k_i}([1^k ,d-k])}
{\prod_{j=1}^{l}b_j!}\bigg|_{d=1}\right\}\right)\mathrm{d}k\\
=& \frac{(-1)^{n-1}}{\prod_{j=1}^{l}b_j!}\int_{-1/2}^{1/2}\left(\frac{\mathrm d^{n-1}}{\mathrm d k^{n-1}}\left\{\prod_{i=1}^{s}\left((1/2-k)^{k_i}-(-k-1/2)^{k_i}\right)\right\}\right)\mathrm{d}k.
\end{aligned}
$$
Next, we observe that
$$
\prod_{i=1}^{s}\left((1/2-k)^{k_i}-(-k-1/2)^{k_i}\right)
$$
is an odd function if $\sum_i k_i-s= 2g+n-1$ (from the dimension constraint) is odd, and
an even function if $\sum_i k_i-s= 2g+n-1$ is even.
Recall that the derivative of an even function is odd, and vice versa.
Consequently, the function under integration
$$
\frac{\mathrm d^{n-1}}{\mathrm d k^{n-1}}\left\{\prod_{i=1}^{s}\left((1/2-k)^{k_i}-(-k-1/2)^{k_i}\right)\right\}
$$
is always even and not zero, whence the integration is not zero.
Hence, the highest degree is exactly $\sum_i k_i -n+1$.
This completes the proof. 
\end{proof}

Remarkably, the difference between the lowest and highest degrees is also $(\sum_i k_i-n+1)-s =2g$.
In analogy with the Witten symbol for single Hurwitz numbers, we define
$$
\langle \langle\tau_{z_1},\ldots,\tau_{z_n},\Lambda_{2g}\rangle \rangle_g ^{[1^{b_1},2^{b_2},\ldots,l^{b_l}]} := (-1)^g[\beta_1^{z_1}\cdots \beta_n^{z_n}]  H_{d,d}(k_1,\ldots,k_s;(d),\beta) Aut(\beta) d.
$$
Now we are ready to present an analogue of the $\lambda_g$ conjecture concerning a compact
formula for the lowest Witten symbols.

\begin{theorem}[Analogue of the $\lambda_g$ conjecture]\label{thm:lambda-g}
	Let $B_k$ be the k-th Bernoulli number. Then, for $z_1+\cdots+z_n =s= \sum_{i = 1}^{l}ib_i -2g-n+1$, we have
	$$
	\begin{aligned}
		&\langle \langle\tau_{z_1},\ldots,\tau_{z_n},\Lambda_{2g}\rangle \rangle_g ^{[1^{b_1},2^{b_2},\ldots,l^{b_l}]} =\binom{z_1+\cdots + z_n}{z_1,\ldots,z_n}   \mathbf{C}_{g}^{[1^{b_1},2^{b_2},\ldots,l^{b_l}]} , \\
	\end{aligned}
	$$
	where for a fixed $K=[1^{b_1},2^{b_2},\ldots,l^{b_l}] ,\, \mathbf{C}_{g}^{[1^{b_1},2^{b_2},\ldots,l^{b_l}]}$ only depends on $g$,
	$$
	\mathbf{C}_{g}^{[1^{b_1},2^{b_2},\ldots,l^{b_l}]} =  \frac{(2^{2g-1}-1) \prod_{i} i^{b_i}\prod_i((i-1)b_i)!}{2^{2g-1}(2g)!\prod_{i=1}^{l}(i!)^{b_i}} |B_{2g}|.
	$$
\end{theorem} 
\begin{proof}
	In the course of determining the lowest degree in Theorem~\ref{thm:poly-double}, we have seen that
	$$
	\begin{aligned}
		&\quad \langle \langle\tau_{z_1},\ldots,\tau_{z_n},\Lambda_{2g}\rangle \rangle_g ^{[1^{b_1},2^{b_2},\ldots,l^{b_l}]} \\
		&=(-1)^g[\beta_1^{z_1}\cdots \beta_n^{z_n}] dH_{d,d}(k_1,\ldots,k_s;(d),\beta) Aut(\beta)\\
		& = 
		(-1)^g\binom{\sum_{i}b_i}{z_1,\ldots,z_n}
		\frac{ \prod_{i} i^{b_i}\prod_i((i-1)b_i)!}{2^{2g}\prod_{i=1}^{l}(i!)^{b_i}} 
		\sum_{\lambda \vdash g } \frac{\xi_{2\lambda} (-1)^{\ell(\lambda)}}{|Aut(\lambda)|} \\
	\end{aligned}
	$$
	Note that it was proved, e.g., in~\cite{3-cycle}, that, 
	$$
	\sum_{\lambda \vdash g } \frac{\xi_{2\lambda} (-1)^{\ell(\lambda)}}{|Aut(\lambda)|} 
	= (-1)^{g}   \frac{2^{2g}-2}{(2g)!} |B_{2g}|  .
	$$
	Plugging it into the last formula will complete the proof.
\end{proof}

Setting $b_r=s$ and $b_i=0$ for all $i\neq r$, we recover the analogue for
double Hurwitz numbers with completed $r$-cycles~\cite{chw}.

\subsection{String and dilaton equations}

In this subsection, we present analogues of the string and dilaton equations
satisfied by the classical Witten symbols.

\begin{proposition}[String equation]
	For $g\geq 0,\,n\geq 1,\,b_1,\ldots,b_l \geq 0 ,\, x\geq 1$, and $z_1+\cdots+z_n =s +1= \sum_{i = 1}^{l}ib_i -2g-n+1+1  $, we have:
	\begin{equation}
    \begin{split}
		&\quad \langle \langle\tau_0^{x-1},\tau_{z_1},\ldots,\tau_{z_n},\Lambda_{2g}\rangle \rangle_g ^{[1^{b_1},2^{b_2},\ldots,x^{b_x+1},\ldots,l^{b_l}]}\\
		&= \frac{((x-1)(b_{x}+1))_{x-1}}{(x-1)!} \sum_{i=1}^{n}\langle \langle\tau_{z_1},\ldots,\tau_{z_{i-1}},\tau_{z_i+1},\tau_{z_{i+1}},\ldots,\tau_{z_n},\Lambda_{2g}\rangle \rangle_g ^{[1^{b_1},2^{b_2},\ldots,x^{b_x},\ldots,l^{b_l}]} 
	\end{split}
    \end{equation}  
\end{proposition}
\begin{proof}
Based on Theorem~\ref{thm:lambda-g}, we have
	$$
    \begin{aligned}
		& \quad \langle \langle\tau_0^{x-1},\tau_{z_1},\ldots,\tau_{z_n},\Lambda_{2g}\rangle \rangle_g ^{[1^{b_1},2^{b_2},\ldots,x^{b_x+1},\ldots,l^{b_l}]}  \\
		& =\binom{1+\sum_{i}b_i=s+1}{0^{x-1},z_1,\ldots,z_n}   \mathbf{C}_{g}^{[1^{b_1},2^{b_2},\ldots,x^{b_x+1},\ldots,l^{b_l}]} \\
		& =\binom{\sum_{i = 1}^{l}ib_i -2g-n+1}{0^{x-1},z_1,\ldots,z_n}   \mathbf{C}_{g}^{[1^{b_1},2^{b_2},\ldots,x^{b_x+1},\ldots,l^{b_l}]} \\
		& =\frac{(\sum_{i}b_i=s)!(z_1+\cdots+z_n)}{z_1!\cdots z_n!} \frac{(2^{2g-1}-1) \prod_{i} i^{b_i}\prod_i((i-1)b_i)!}{2^{2g-1}(2g)!\prod_{i=1}^{l}(i!)^{b_i}} |B_{2g}|\frac{x((x-1)(b_{x}+1))_{x-1}}{x!} \\& =\frac{(\sum_{i}b_i=s)!(z_1+\cdots+z_n)}{z_1!\cdots z_n!} \mathbf{C}_{g}^{[1^{b_1},2^{b_2},\ldots,x^{b_x},\ldots,l^{b_l}]} \frac{x((x-1)(b_{x}+1))_{x-1}}{x!} \\
		& = \frac{((x-1)(b_{x}+1))_{x-1}}{(x-1)!} \sum_{i=1}^{n}\langle \langle\tau_{z_1},\ldots,\tau_{z_{i-1}},\tau_{z_i+1},\tau_{z_{i+1}},\ldots,\tau_{z_n},\Lambda_{2g}\rangle \rangle_g ^{[1^{b_1},2^{b_2},\ldots,x^{b_x},\ldots,l^{b_l}]},
    \end{aligned}
    $$
    and the proof follows.
\end{proof}

\begin{proposition}[Dilaton equation]
	For $g\geq 0,\, n\geq 1$, and $ z_1,\ldots,z_n \geq 0 ,\, x\geq 1$, such that $z_1+\cdots+z_n  =s= \sum_{i = 1}^{l}ib_i -2g-n+1$, we have
	\begin{equation}
     \begin{split}
		&\quad \langle \langle\tau_0^{x-2},\tau_1,\tau_{z_1},\ldots,\tau_{z_n},\Lambda_{2g}\rangle \rangle_g ^{[1^{b_1},2^{b_2},\ldots,x^{b_x+1},\ldots,l^{b_l}]}\\
		&= \frac{(s+1)((x-1)(b_{x}+1))_{x-1}}{(x-1)!}\langle \langle\tau_{z_1},\ldots,\tau_{z_n},\Lambda_{2g}\rangle \rangle_g ^{[1^{b_1},2^{b_2},\ldots,l^{b_l}]}
	\end{split}
    \end{equation}
\end{proposition}
\begin{proof}
	We compute
	$$
	\begin{aligned}
		&\quad  \langle \langle\tau_0^{x-2},\tau_1,\tau_{z_1},\ldots,\tau_{z_n},\Lambda_{2g}\rangle \rangle_g ^{[1^{b_1},2^{b_2},\ldots,x^{b_x+1},\ldots,l^{b_l}]}\\
		& =\binom{1+\sum_{i}b_i=s+1}{0^{x-1},1,z_1,\ldots,z_n}   \mathbf{C}_{g,n+x-1}^{[1^{b_1},2^{b_2},\ldots,x^{b_x+1},\ldots,l^{b_l}]} \\
		& =\binom{\sum_{i = 1}^{l}ib_i -2g-n+1}{0^{x-1},1,z_1,\ldots,z_n}   \mathbf{C}_{g,n+x-1}^{[1^{b_1},2^{b_2},\ldots,x^{b_x+1},\ldots,l^{b_l}]} \\
		& = \binom{s=\sum_{i = 1}^{l}b_i}{z_1,\ldots,z_n}  \mathbf{C}_{g,n}^{[1^{b_1},2^{b_2},\ldots,l^{b_l}]} \frac{(s+1)((x-1)(b_{x}+1))_{x-1}}{(x-1)!}\\
		& = \frac{(s+1)((x-1)(b_{x}+1))_{x-1}}{(x-1)!}\langle \langle\tau_{z_1},\ldots,\tau_{z_n},\Lambda_{2g}\rangle \rangle_g ^{[1^{b_1},2^{b_2},\ldots,l^{b_l}]}
	\end{aligned}
	$$
	and the proof follows.
\end{proof}


\newpage


\section*{Acknowledgements}

We thank Zi-Wei Bai for valuable discussions.

%
%
%
%


\end{document}